\documentclass[12pt]{amsart}

\usepackage{amssymb}

\usepackage{comment}
\usepackage{amsmath} 
\usepackage{amscd}
\usepackage[matrix,arrow]{xy}
\usepackage[OT2,T1]{fontenc}

\makeatletter
\@namedef{subjclassname@2010}{%
  \textup{2010} Mathematics Subject Classification}
\makeatother

\newtheorem{thm}{Theorem}
\newtheorem{lem}{Lemma}[section]
\newtheorem{prop}[lem]{Proposition}
\newtheorem{cor}[lem]{Corollary}

\theoremstyle{definition}
\newtheorem{rem}[lem]{Remark}

\numberwithin{equation}{section}

\DeclareSymbolFont{cyrletters}{OT2}{wncyr}{m}{n}
\DeclareMathSymbol{\Sha}{\mathalpha}{cyrletters}{"58}
\DeclareMathOperator{\Sel}{Sel}

\DeclareMathOperator{\rank}{rank}

\DeclareMathOperator{\Ann}{Ann}
\DeclareMathOperator{\Rad}{Rad}
\DeclareMathOperator{\Gal}{Gal}
\DeclareMathOperator{\Norm}{Norm}
\DeclareMathOperator{\Aut}{Aut}
\DeclareMathOperator{\ord}{ord}

\newcommand{\p}{\mathfrak{P}}
\newcommand{\Q}{{\mathbb Q}}
\newcommand{\Z}{{\mathbb Z}}

\newcommand{\F}{{\mathbb F}}

\newcommand{\cA}{\mathcal{A}}
\newcommand{\cB}{\mathcal{B}}

\newcommand{\OO}{{\mathcal O}}

\DeclareMathOperator{\Jac}{Jac}

\begin{document}


\baselineskip=17pt



\title[Perfect powers expressible as $x^5+y^5$ or $x^7+y^7$]{Perfect powers expressible as sums of two fifth or seventh powers}

\author[S. R. Dahmen]{Sander R. Dahmen}
\address{
Department of Mathematics\\
VU University Amsterdam\\
De Boelelaan 1081a\\
1081 HV Amsterdam\\
The Netherlands}
\email{s.r.dahmen@vu.nl}

\author[S. Siksek]{Samir Siksek}
\address{Mathematics Institute\\
University of Warwick\\
Coventry\\
CV4 7AL\\
United Kingdom}
\email{s.siksek@warwick.ac.uk}

\date{January 25, 2014}

\begin{abstract}
We show that the generalized Fermat equations with signatures $(5,5,7)$, $(5,5,19)$, and $(7,7,5)$ (and unit coefficients) have no non-trivial primitive integer solutions. 
Assuming GRH, we also prove the nonexistence of non-trivial primitive integer solutions for the signatures $(5,5,11)$, $(5,5,13)$, and $(7,7,11)$. The main ingredients for obtaining our results are descent techniques, the method of Chabauty-Coleman, and the modular approach to Diophantine equations.
\end{abstract}



\maketitle

\section{Introduction}

Let $p$, $q$, $r \in \Z_{\geq 2}$. The equation
\begin{equation}\label{eqn:FCgen}
x^p+y^q=z^r
\end{equation}
is known as the Generalized Fermat equation (or the Fermat--Catalan equation)
 with signature $(p,q,r)$ (and unit coefficients).
As in Fermat's Last Theorem, one is interested in integer solutions
$x$, $y$, $z$. Such a solution is called {\em non-trivial} if
$xyz \neq 0$, and {\em primitive} if $x$, $y$, $z$ are coprime.
Let $\chi=p^{-1}+q^{-1}+r^{-1}$. The parametrization
of non-trivial primitive integer solutions for $(p,q,r)$ with
$\chi \geq 1$ has now been completed \cite{Ed}.
The Generalized Fermat Conjecture \cite{Da97}, \cite{DG}
is concerned with the case $\chi<1$.
It states that the only non-trivial primitive integer solutions to~\eqref{eqn:FCgen} with $\chi<1$ are given by
\begin{gather*}
1+2^3 = 3^2, \quad 2^5+7^2 = 3^4, \quad 7^3+13^2 = 2^9, \quad
2^7+17^3 = 71^2, \\
3^5+11^4 = 122^2, \quad 17^7+76271^3 = 21063928^2, \quad
1414^3+2213459^2 = 65^7, \\
9262^3+15312283^2 = 113^7, \quad
43^8+96222^3 = 30042907^2, \\ 33^8+1549034^2 = 15613^3.
\end{gather*}
The Generalized Fermat Conjecture
has been established for many signatures $(p,q,r)$,
including for several infinite families of signatures.
For exhaustive surveys see Cohen's book \cite[Chapter 14]{Cohen},
or \cite{BeChDaYa}. 

Many of the equations are solved using the modular approach to Diophantine
equations. If we restrict ourselves to Frey curves over $\Q$ and the signature
$(p,q,r)$ with $\chi<1$ consisting of only primes, then the only signatures (up
to permutation) for which a Frey curve is known are given by 
\[
(l,l,l), \;
(l,l,2), \; (l,l,3),\; (2,3,l),\; (3,3,l),\;
(5,5,l),\; (7,7,l)
\]
where $l$ is a prime ($\geq 5,5,5,7,5,2,2$
respectively to ensure that $\chi<1$). These Frey
curve are all already mentioned in \cite{Da97}.
For all but the last two signatures, these Frey
curves have been used to completely solve at least 
one Generalized Fermat equation (with unit
coefficients, as always throughout this paper).
In fact, the first three cases have completely
been solved. 
The $(l,l,l)$ case corresponds of course to
Fermat's Last Theorem \cite{Wiles95} (with
exponent $l\geq 5$) and the $(l,l,2)$ and
$(l,l,3)$ cases have been solved for $l\ \geq 7$
by Darmon and Merel \cite{DarmonMerel97} using a
modular approach and for $l=5$ by Poonen
\cite{Poonen98} using descent on elliptic curves
and Jacobians of genus $3$ cyclic covers of the
projective line.
The $(2,3,l)$ case has only been solved (recall
that we have now restricted ourselves to primes $l\geq 7$) for
$l=7$ using a combination of the modular approach and
explicit methods (including Chabauty-Coleman) for
determining $\Q$-rational points on certain genus
$3$ curves (twists of the Klein quartic); see
\cite{PSS}. Finally the $(3,3,l)$ case is solved
for a set of prime exponents $l$ with Dirichlet
density $\geq 0.628$, and all $l \leq 10^9$; see
\cite{ChenS}.  One feature that is common to the 
Frey curves associated to the first five
signatures, is that evaluating the Frey curve at a
trivial solution gives either a singular curve or
an elliptic curve with complex multiplication.
This is one of
the main reasons why the first three signatures
can be dealt with for all relevant prime exponents and why
in the $(3,3,l)$ case so many prime exponents $l$ can be
handled. In the latter case the main obstruction
to solving the equation completely is because of
the Catalan solution (which is actually only present for $l=2$, but nevertheless still forms an obstruction for larger primes $l$). The Catalan solution also forms an
obstruction for the $(2,3,l)$ case, but here there
are many other difficulties.

The main reason why the Frey curves associated to
signature $(5,5,l)$ or $(7,7,l)$ have not been used
before to completely solve a generalized Fermat
equation, is probably because evaluating the Frey
curve at a (primitive) trivial integer solution
does not always give a singular or CM curve. In
fact, only $(\pm 1)^5+(\mp 1)^5=0^l$ leads to a
singular curve (throughout this paper, when $\pm$ or $\mp$ signs are present in a formula, they are meant to correspond in the obvious way within the formula).
The modular approach however still
gives a lot of non-trivial information. This
allows us to combine the modular approach with the
method of Chabauty-Coleman and descent techniques
to solve three new cases of the generalized Fermat
equations
\begin{align}
\label{eqn:55l} x^5+y^5 = & z^l\\
\label{eqn:77l} x^7+y^7 = & z^l.
\end{align}
In fact, the only values for which these equations already have been solved, are  covered by the first three families of exponent triples: $(5,5,2)$ and $(5,5,3)$ are solved by Poonen, $(7,7,2)$ and $(7,7,3)$ are solved by Darmon and Merel, while the cases $(5,5,5)$ and $(7,7,7)$ are of course special cases of Fermat's Last Theorem, and the modular method using a Frey curve for exponent $(l,l,l)$ works for these two special cases as well (of course there are classical descent proofs, exponent $5$ was first solved around 1825 independently by Legendre and Dirichlet, exponent $7$ was first solved around 1839 by Lam\'e).
We see that the first two open cases for~\eqref{eqn:55l} and~\eqref{eqn:77l} are the signatures $(5,5,7)$ and $(7,7,5)$ respectively. In this paper we shall solve these equations, as well as the equation with signature $(5,5,19)$. 

\begin{thm}\label{thm:55l}
Let $l=7$ or $l=19$. Then the only solutions to the equation
\[x^5+y^5=z^l\]
in coprime integers $x$, $y$, $z$ are $(\pm 1,\mp 1,0)$, $(\pm 1, 0, \pm 1)$, and $(0,\pm 1, \pm 1)$.
\end{thm}

\begin{thm}\label{thm:775}
The only solutions to the equation
\begin{equation}\label{eqn:775}
x^7+y^7=z^5
\end{equation}
in coprime integers $x$, $y$, $z$ are $(\pm 1,\mp 1,0)$, $(\pm 1, 0, \pm 1)$, and $(0,\pm 1, \pm 1)$.
\end{thm}

To prove Theorem~\ref{thm:55l}, we exploit the fact that the associated Frey curve evaluated at a primitive trivial integer solution with $z=0$ gives a singular curve in order to solve~\eqref{eqn:55l} when $5 \mid z$ for all primes $l$. For the remaining case $5 \nmid z$, we
relate primitive integer solutions of~\eqref{eqn:55l} with $l=7$ and $l=19$ 
to $\Q$-rational points on a curve of genus $3$ and $9$ respectively, which we are able to determine using Chabauty-Coleman.
For Theorem~\ref{thm:775}, we relate primitive integer solutions of~\eqref{eqn:775}
to $K$-rational points on genus $2$ curves over the totally real cubic
field $K=\Q(\zeta+\zeta^{-1})$ where $\zeta$ is a primitive $7$-th
root of unity. Our factorization argument leads us in fact to $50$ $5$-tuples of
such genus $2$ curves for which we need to determine the $K$-rational
points for at least one curve per $5$-tuple. We shall use the modular approach
to rule out all but two of the $5$-tuples of genus $2$ curves.
For the remaining two $5$-tuples of curves, we were able to determine enough $K$-rational
points using the method of Chabauty-Coleman to finish the proof of Theorem~\ref{thm:775}.
We used the computer package {\tt MAGMA} \cite{MAGMA} for all our calculations. The {\tt MAGMA} scripts we refer to in this paper are posted at \verb|www.few.vu.nl/~sdn249/sumsofpowers.html|.

Many of our computations depend on class group and unit group computations, which become significantly faster under assumption of the generalized Riemann hypothesis for Dedekind zeta functions (abbreviated as GRH from now on). As it turns out, assuming GRH, we can also deal with the exponents $(5,5,11), (5,5,13)$, and $(7,7,11)$.

\begin{thm}\label{thm:GRH}
Assume GRH. If $l \in \{11,13\}$, then~\eqref{eqn:55l} has no non-trivial primitive integer solutions. If $l=11$, then~\eqref{eqn:77l} has no non-trivial primitive integer solutions.
\end{thm}

\section{Preliminaries}

\subsection{The method of Chabauty-Coleman}\label{sec:chab}

Chabauty-Coleman is a method for bounding
the number of $K$-rational points on a curve of genus $\ge 2$ defined over a number field $K$,
subject to certain conditions. We will need Chabauty-Coleman
for the proof of our Theorems~\ref{thm:55l}, \ref{thm:775}, and \ref{thm:GRH}, and so we provide
in this section a brief sketch of the method. For details
we recommend the expository paper of  
McCallum and Poonen \cite{MP}, as well as Wetherell's thesis \cite{We},
and Coleman's original paper \cite{Co}.

Let $C/K$ be a smooth projective geometrically integral curve of genus
$g \geq 2$, and let $J$ be its Jacobian.
It is convenient to suppose the 
existence of $K$-rational points on $C$ and fix one such point
$P_0 \in C(K)$. We use $P_0$ as the base for our
Abel-Jacobi embedding:
\[
\jmath : C \rightarrow J, \qquad
P \mapsto [P-P_0].
\]
Let $\p$ be a prime of good reduction for $C$ and denote by $K_{\p}$ the $\p$-adic completion of $K$.
Write $\Omega(C/K_{\p})$ for the $K_{\p}$-vector space of regular differentials on
$C$, and $\Omega(J/K_{\p})$ for the corresponding space on $J$. Both these
spaces have dimension $g$, and the Abel-Jacobi embedding induces
an isomorphism $\jmath^* : \Omega(C/K_{\p}) \rightarrow \Omega(J/K_{\p})$;
this is independent of the choice of base point $P_0$, and we shall use
it to identify the two spaces.

The method of Chabauty is based on the integration pairing
\begin{equation}\label{eqn:pairing}
\Omega(C/K_{\p}) \times J(K_{\p}) \rightarrow K_{\p},
\qquad (\omega,D) \mapsto \int_0^D \omega.
\end{equation}
The Mordell-Weil group $J(K)$ is contained in $J(K_{\p})$.
Let $r$ be its rank, and write $\Ann(J(K))$ for
the $K_{\p}$-subspace of $\Omega(C/K_{\p})$ that
annihilates $J(K)$ in the above pairing.
If $r<g$, then this has dimension at least $g-r$.
Suppose $\Ann(J(K))$ is positive dimensional
and let $\omega$ be a non-zero differential
belonging to it. Denote by $\F_{\p}$ the residue
class field of $K_{\p}$, let $p$ be its
characteristic and let $e$ denote the absolute
ramification index of $\p$. We scale $\omega$
so that it reduces to a non-zero differential
$\overline{\omega}$ on the reduction $\tilde{C}/\F_{\p}$. The differential $\omega$ can be used to bound
the number of $K$-rational points $C(K)$. In particular, if $\overline{\omega}$ does not vanish at $\overline{P} \in \tilde{C}(\F_{\p})$ and $e<p-1$, then there is at most one $K$-rational point $P$ on $C$
that reduces to $\overline{P}$ modulo $\p$.

\begin{rem}
In \cite{Siksek} a modified version of the above method is developed where instead of the traditional $r \le g-1$
condition of Chabauty-Coleman the necessary
condition of the new method is $r \le [K:\Q] (g-1)$. However, as it turns out, the \lq classical\rq\ Chabauty-Coleman method sketched above suffices for our purposes.
\end{rem}

\subsection{The modular approach}

Our proofs will make heavy use of the modular approach to Diophantine equations, involving Frey curves, modularity, Galois representations and level-lowering. For an introduction, the reader can consult e.g.~\cite[Chapter 15]{Cohen} or~\cite[Chapter 2]{Dahmen}. 
By a newform of level $N$ we will mean a cuspidal newform of weight $2$ with respect to $\Gamma_0(N)$ (so the character is trivial). A newform is always normalized by default (i.e. the first Fourier coefficient of the expansion at the infinite cusp equals $1$).

\subsection{A standard factorization lemma}

The following simple result will be very useful when we are factorizing $x^5+y^5$ and $x^7+y^7$.

\begin{lem} \label{lem:gcd}
Let $p$ be an odd prime and $x$, $y$ coprime integers. Write 
\[
H_p=\frac{x^p+y^p}{x+y}.
\]
Then $g:=\gcd(x+y,H_p)=1$ or $p$. Consequently, $g=p \Leftrightarrow p|x^p+y^p \Leftrightarrow p|H_p \Leftrightarrow p|x+y$.
Moreover, $p^2 \nmid H_p$.
\end{lem}

\begin{proof}
Let $u=-(x+y)$, then using the binomial formula we get
\[
H_p=\frac{(y+u)^p-y^p}{u}=\sum_{k=1}^p 
\binom{p}{k}
u^{k-1} y^{p-k}.
\]
Form the expression above we see that $g|py^{p-1}$. Since $\gcd(u,y)=\gcd(x,y)=1$, we get $g|p$.
Furthermore, if $p \nmid u$, then using that 
$p|\binom{p}{k}
$ for $k=1,\ldots, p-1$ we see that $p \nmid H_p$. If $p|u$, then $H_p \equiv p y^{p-1} \pmod{p^2}$, which is nonzero modulo $p^2$ since $p\nmid y$.
\end{proof}

\section{Proof of Theorem~\ref{thm:55l}}\label{sec:55l}

In light of Lemma~\ref{lem:gcd} it is natural to
distinguish two cases. Namely non-trivial
primitive integer solutions to \eqref{thm:55l} with
$5 \nmid z$ on the one hand and those with $5 | z$
on the other hand. For the former case, we relate 
non-trivial
primitive integer solutions to \eqref{eqn:55l} for
some odd prime $l$ to determining 
$\Q$-rational points on the hyperelliptic curve
\begin{equation}\label{eqn:Cl}
C_l: Y^2=20X^l+5.
\end{equation}
Note that this curve has genus $(l-1)/2$ and that
\[
C_l(\Q) \supset \{\infty,(1,\pm 5)\}.
\]

\begin{lem}\label{lem:ReductionToCurves55lCaseI}
Let $l$ be an odd prime. If
\begin{equation}\label{eqn:RatPoints55lCaseI}
C_l(\Q)=\{\infty,(1,\pm 5)\}
\end{equation}
then there are no non-trivial primitive integer solutions to \eqref{eqn:55l} with $5\nmid z$
\end{lem}

\begin{proof}
Suppose that $x$, $y$, $z$ are non-zero coprime integers satisfying \eqref{eqn:55l}. Then
\begin{equation}\label{eqn:55lfactors}
(x+y)H_5=z^l.
\end{equation}
For any odd prime $p$, we have that $H_p$ is a symmetric binary form of even degree in $x, y$, hence a binary form in $x^2+y^2$ and $(x+y)^2$. For $p=5$ we have explicitly
\begin{equation}\label{eqn:identity}
5(x^2+y^2)^2=4 H_5+(x+y)^4.
\end{equation}
We assume $5 \nmid z$.
By Lemma~\ref{lem:gcd} we have $\gcd(x+y,H_5)=1$. Hence \eqref{eqn:55lfactors} yields
\begin{equation}\label{eqn:Factor55lCaseI}
x+y=z_1^l, \qquad H_5= z_2^l
\end{equation}
where $z_1$, $z_2$ are coprime non-zero integers satisfying $z=z_1 z_2$. 
Using identity~\eqref{eqn:identity} we have
\[
5 (x^2+y^2)^2=4 z_2^l+z_1^{4l}.
\]
Multiplying both sides by $5/z_1^{4l}$, we see that
\[P:=\left(\frac{z_2}{z_1^4}, \frac{5(x^2+y^2)}{z_1^{2l}}  \right) \in C_l(\Q).\]
Since $z_1 \neq 0$, we have $P \neq \infty$. If $P=(1,\pm 5)$, then we see that $z_2=1$ and $z_1=\pm 1$, which by~\eqref{eqn:Factor55lCaseI} leads to $xy=0$. A contradiction which proves the lemma.
\end{proof}

We expect that \eqref{eqn:RatPoints55lCaseI} holds for all primes $l \geq 7$ (it holds for $l=5$, but we do not need this here). The cases we can prove at the present are summarized as follows.

\begin{prop}\label{prop:RatPoints55lCaseI}
If $l=7$ or $l=19$, then
\[C_l(\Q)=\{\infty,(1,\pm 5)\}.\]
\end{prop}

A proof, using $2$-descent on hyperelliptic Jacobians and the method of Chabauty-Coleman, is given in Section~\ref{sec:Chabauty55l} below. In a similar fashion we can reduce proving the nonexistence of non-trivial primitive integer solutions with $5 | z$
(to \eqref{eqn:55l} for some odd prime $l$) to finding $\Q$-rational points on a twist of $C_l$; see Section~\ref{sec:Discussion}.
We can however deal with this case in a uniform manner for all primes $l \geq 7$ using the modular method; see Section~\ref{sec:Modular55l}. Taking into account previously solved small exponent cases, we have in fact a complete solution in the $5|z$ case.

\begin{prop}\label{prop:SecondCase55l}
Let $l\geq 2$ be an integer. There are no non-trivial primitive integer solutions to~\eqref{eqn:55l} with $5 | z$
\end{prop}

Trivially, Lemma \ref{lem:ReductionToCurves55lCaseI} and Propositions \ref{prop:RatPoints55lCaseI} and \ref{prop:SecondCase55l} together imply Theorem~\ref{thm:55l}.

\subsection{Rational points on $C_l$}\label{sec:Chabauty55l}

Let $J_l$ denote the Jacobian of $C_l$ and $g_l=(l-1)/2$ the genus of $C_l$ (which equals the dimension of $J_l$).
In order to use Chabauty-Coleman to determine the $\Q$-rational points on $C_l$ for some $l$, it is necessary that the Chabauty condition, $\rank J_l(\Q) < g_l$, is satisfied and we need to compute a subgroup 
of finite index in the Mordell-Weil group $J_l(\Q)$.

Before we go into the rank computations, we start with a description of  the torsion subgroup $J_l(\Q)_\mathrm{tors}$ of $J_l(\Q)$. The curve $C_l$, and hence its Jacobian $J_l$, has good reduction away from $2$, $5$, $l$.
For any odd prime $p$ of good reduction, the natural map
\[
J_l(\Q)_\mathrm{tors} \rightarrow J_l(\F_p)
\]
is injective. In the rest of this section, $l$ will always stand for a prime in the range $7 \leq l \leq 19$.
Using {\tt MAGMA} we find for every prime $l$ in our range, two primes $p_1\not=p_2$ distinct from $2,5,$ or $l$ such that
\[
\gcd(\# J_l(\F_{p_1}), \# J_l(\F_{p_2}))=1.
\]
This shows that for all these primes $l$ we have $J_l(\Q)_\mathrm{tors}=\{0\}$. To be concrete, for $l=7,11,13,17,19$ we can take $(p_1,p_2)=(3,43)$, $(13,23)$, $(3,53)$, $(3,103)$, $(7,191)$ respectively.

As for the rank computations, {\tt MAGMA} includes implementations by Nils Bruin and Michael Stoll
of $2$-descent on Jacobians of hyperelliptic curves over number fields;
the algorithm is detailed in Stoll's paper \cite{Stoll}. Using this we were
able to compute the $2$-Selmer ranks of $J_l/\Q$ for the primes $l$ in our
range (and no further, not even assuming GRH). The values are given in Table
\ref{table:RankBoundsC} below together with the time it took to compute them
on a machine with 2 Intel Xeon dual
core CPUs at 3.0 GHz.
We want to stress that the {\tt MAGMA} routine {\tt TwoSelmerGroup} involved,
makes use of the pseudo-random number generator of {\tt MAGMA}. So the exact
time also depends on the seed. From the usual exact sequence 
\[
0 \to
J_l(\Q)/2J_l(\Q) \to \Sel^{(2)}(\Q,J_l) \to \Sha(\Q,J_l)[2] \to 0
\]
together with the fact that $J_l(\Q)$ has no 2-torsion, we get
\begin{equation}\label{eqn:SelmerBound}
\rank J_l(\Q) = \dim_{\F_2}\Sel^{(2)}(\Q,J_l)-\dim_{\F_2}\Sha(\Q,J_l)[2] \leq \dim_{\F_2}\Sel^{(2)}(\Q,J_l).
\end{equation}
Let $D=[(1,5)-\infty]$, then $D$ is a non-zero element of $J_l(\Q)$ and therefore (remembering that $J_l(\Q)_\mathrm{tors}=\{0\}$)
has infinite order. This shows that
\[\rank J_l(\Q) \geq 1.\]
In particular, we get from the $2$-Selmer ranks of $J_l/\Q$ in Table \ref{table:RankBoundsC} that for $l=7,19$ we have
$\rank J_l(\Q)=1$ and $D$ generates a subgroup of finite index in $J_l(\Q)$. 

\begin{table}[h!]
\caption{Rank bounds for the Jacobian of $C_l$}
\begin{tabular}{c|c|c}
$l$ & $\dim_{\F_2}\Sel^{(2)}(\Q,J_l)$ & Time \\
\hline
7 & 1 & 0.4s \\
11 & 2 & 3s \\
13 & 2 & 23s \\
17 & 2 & 4821s $\approx$ 1.3h \\
19 & 1 & 109819s $\approx$ 30.5h \\
\end{tabular}
\label{table:RankBoundsC}
\end{table}

\begin{rem}
Assume that $\Sha(\Q,J_l)$ is finite. As $C_l(\Q) \ne \emptyset$,
it follows from the work of Poonen and Stoll \cite{PS}
that the Cassels-Tate pairing on $\Sha(\Q,J_l)$
is alternating, and so $\# \Sha(\Q,J_l)$ is a square.
In this
case, we get from the equality in~\eqref{eqn:SelmerBound} that $\rank J_l(\Q)$
and $\dim_{\F_2}\Sel^{(2)}(\Q,J_l)$ have the same parity. Together with the
fact that $\rank J_l(\Q) \geq 1$ we now see that we can read of
$\rank J_l(\Q)$ from Table~\ref{table:RankBoundsC} (still assuming the
finiteness of $\Sha(\Q,J_l)$, or actually just of the $2$-part).  \end{rem}

\begin{rem}
Instead of using a $2$-descent on $J_l/\Q$, we can also apply \cite{StollI},
\cite{StollII} to get an upper bound for $\rank J_l(\Q)$ using a $(1-\zeta_l)$-descent 
on $J_l/\Q(\zeta_l)$. It turns out that for $l=7,11$ this gives the
same upper bound for $\rank J_l(\Q)$ as given by Table~\ref{table:RankBoundsC}
(namely $1$ and $2$ respectively). For $l=13,17,19$ however, the upper bounds
obtained from a $(1-\zeta_l)$-descent are strictly larger than the bounds
given by Table~\ref{table:RankBoundsC}.  \end{rem}

For $l=7$, $19$ both the Chabauty condition is satisfied and we have
explicitly found a subgroup of finite index in the Mordell-Weil group
$J_l(\Q)$.  We are thus in a position to use the method of Chabauty-Coleman to
determine $C_l(\Q)$ for these $l$.

\begin{proof}[Proof of Proposition~\ref{prop:RatPoints55lCaseI}]
Let $l \in \{7,19\}$ and let $J_l$ denote, as before, the Jacobian of $C_l$. We
already know that $\rank J_l(\Q)=1$ and $D:=[(1,5)-\infty] \in J_l(\Q)$
generates a subgroup of finite index in $J_l(\Q)$.  We shall apply the method
of Chabauty-Coleman, sketched in Section~\ref{sec:chab}, with $p=3$. A basis
for $\Omega(C_l/\Q_3)$ is given by $X^i \frac{dX}{Y}$ with $i=0, 1, \ldots,
g_l-1=(l-3)/2$. 
For $l=7$ we find
\[\int_0^D  \frac{dX}{Y}\equiv 3\cdot 40, \qquad \int_0^D  X\frac{dX}{Y}\equiv 3^2 \cdot 25, \qquad \int_0^D  X^2 \frac{dX}{Y}\equiv 3^2 \cdot 13 \pmod{3^5}.\]
For $l=19$ we find the following congruences modulo ${3^5}$,
\[\left(\int_0^D X^k \frac{dX}{Y}\right)_{k=0}^8 \equiv (3\cdot 43, 3 \cdot 76, 3 \cdot 16, 3 \cdot 22, 3 \cdot 65, 3 \cdot 74, 3^2 \cdot 17, 3^2 \cdot 23, 3^2 \cdot 22 );\]
for hints on the evaluation of $p$-adic integrals see \cite{We} (especially
Section 1.9).  Using these values, one can easily approximate
 an explicit basis for
$\Ann(J_l(\Q))$ in both cases. However, it is enough to notice that
\[
\ord_3\left(\int_0^D \frac{dX}{Y}\right)=1, \qquad \ord_3\left(\int_0^D X^{g_l-1} \frac{dX}{Y}\right)=2.
\]
Thus we can find some $\omega_l \in \Ann(J_l(\Q))$ of the form 
\[\omega_l=3\alpha_l \frac{dX}{Y}+ X^{g_l-1} \frac{dX}{Y}, \quad \alpha_l \in \Z_3, \quad \ord_3(\alpha_l)=0.\]
We reduce to obtain a differential on $\tilde{C_l}/\F_3$,
\[\overline{\omega}_l = X^{g_l-1} \frac{dX}{Y}.\]
The differential $\overline{\omega}_l$ does not vanish at any of the four points of $C_l(\F_3)$:
\[
C_l(\F_3)=\{\overline{\infty},(\overline{1},\overline{1}),(\overline{1},\overline{2}),
(\overline{2},\overline{0})\}.
\]
It follows that for each $\tilde{P} \in C_l(\F_3)$ 
there is at most one  $P \in C_l(\Q)$ that reduces to $\tilde{P}$. Now the rational points
$\infty$, $(1,5)$ and $(1,-5)$ respectively reduce to
$\overline{\infty}$, $(\overline{1},\overline{2})$, $(\overline{1},\overline{1})$.
To complete the proof it is sufficient to show that no $\Q$-rational
point reduces to $(\overline{2},\overline{0})$.
One way of showing this is to use the Mordell-Weil sieve \cite{BS3}. Here is
a simpler method. Note that $(\overline{2},\overline{0})$ lifts
to $(\gamma,0) \in C_l(\Q_3)$ where $\gamma$
is the unique element in $\Q_3$ satisfying $\gamma^l=-1/4$. 
Now the divisor $D^\prime=(\gamma,0)-\infty$ has order
$2$ in $J_l(\Q_p)$, and hence belongs to the left-kernel of the pairing~\eqref{eqn:pairing}.
If there is a $\Q$-rational point that reduces to $(\overline{2},\overline{0})$
then that would force $\overline{\omega}_l$ to vanish at $(\overline{2},\overline{0})$.
This completes the proof. Further details can be found in our {\tt MAGMA} script {\tt Chabauty55l.m}.
\end{proof}

\subsection{A modular approach to $x^5+y^5=z^l$ when $5 \mid z$}\label{sec:Modular55l}

The purpose of this section is to give a proof of Proposition~\ref{prop:SecondCase55l}. Let $(x,y,z)$ be a primitive integer solution to~\eqref{eqn:55l} with $z\not=0$ for some prime $l \geq 7$ and assume $5 | z$.
In this case Lemma~\ref{lem:gcd} gives us $\gcd(x+y,H_5)=5$ and $5^2 \nmid H_5$. Hence \eqref{eqn:55lfactors} yields
\[
5(x+y)=z_1^l, \qquad H_5=5 z_2^l
\]
where $z_1$, $z_2$ are coprime non-zero integers satisfying $z=z_1 z_2$.

Kraus \cite[pp. 329--330]{Kraus99} constructed a Frey curve for
the equation $x^5+y^5=z^l$. Following Kraus,
we associate to our solution $(x,y,z)$ to~\eqref{eqn:55l}
the Frey elliptic curve
\[E'_{x,y}: \; Y^2= X^3+5(x^2+y^2) X^2 + 5 H_5(x,y) X.\]
Since we are assuming that $5|z$ we have that 
$5|H_5(x,y)$. So the quadratic twist over $\Q(\sqrt{-5})$ given by the following model has integer coefficients.
\[E_{x,y}: Y^2=X^3-(x^2+y^2)X^2+\frac{H_5(x,y)}{5} X.\]

We record some of the invariants of $E_{x,y}$:
\begin{align*}
c_4 &=2^4 \cdot 5^{-1} \cdot (2x^4 + 3x^3y + 7x^2y^2 + 3xy^3 + 2y^4) \in\Z, \\
c_6 &=2^5\cdot 5^{-1} \cdot (x^2 + y^2)(x^4 + 9x^3y + 11x^2y^2 + 9xy^3 + y^4) \in \Z,\\
\Delta &=2^4 \cdot 5^{-3} \cdot (x+y)^4 H_5^2=2^4 \cdot 5^{-5} \cdot \left(z_1^2 z_2\right)^{2l} \in \Z,\\
j  & = \frac{2^8 \cdot (2x^4 + 3x^3y + 7x^2y^2 + 3xy^3 + 2y^4)^3}{(x+y)^4 H_5^2}.
\end{align*}

\begin{lem}\label{lem:DiscAndN55l}
The conductor $N$ and minimal discriminant $\Delta_{\mathrm{min}}$ of $E_{x,y}$ satisfy
\begin{itemize}
\item $N=2^\alpha 5 \Rad_{\{2,5\}}(z)$ where $\alpha \in \{1,3,4\}$ and $\Rad_{\{2,5\}}(z)$
is the product of the distinct primes not equal to $2$ or $5$ dividing $z$;
\item  If $2 \nmid z$, then $\Delta_{\mathrm{min}}=\Delta$ and if $2 \mid z$, then $\Delta_{\mathrm{min}}=\Delta/2^{12}$.
\end{itemize}
\end{lem}

\begin{proof} 
Recall that $x$, $y$ are coprime. The resultant of $x^5+y^5$ and $2x^4 + 3x^3y + 7x^2y^2 + 3xy^3 + 2y^4$
is $5^5$. Thus any prime $p \ne 2$, $5$ dividing $z$ cannot divide $c_4$ and divides $\Delta$,
and must therefore be a prime of multiplicative reduction. Using $5|z$, we see that $5 \mid \Delta$ and $5 \nmid c_4$. So $5$ is also a prime of multiplicative reduction. Thus the conductor $N$ is $2^\alpha 5 \Rad_{\{2,5\}}(z)$ for some $\alpha \in \Z_{\geq 0}$.
We also see that the model for $E_{x,y}$ is minimal at any prime $p \ne 2$.

If $2 \nmid z$, then $\ord_2(\Delta)=4$, $\ord_2(c_6)=5$, and $\ord_2(c_4) \geq 5$. So in this case the model for $E_{x,y}$ is minimal at $2$ and a straightforward application of Tate's algorithm \cite[Section IV.9]{SilvII} gives $\alpha \in \{3,4\}$. Finally, if $2 \mid z$, then 
 $\ord_2(\Delta) \geq 32$ and $\ord_2(c_4)=4$. A straightforward application of Tate's algorithm shows that the model for $E_{x,y}$ is not minimal at $2$ and we get a new model $E'$ that is integral at $2$ with $\ord_2(\Delta')=\ord_2(\Delta)-12 \geq 20$ and $\ord_2(c_4')=\ord_2(c_4)-4=0$. So in this case $E_{x,y}$ has multiplicative reduction at $2$ and $\Delta_{\mathrm{min}}=\Delta/2^{12}$.
\end{proof}

For a prime $l$ we write $\rho_l^{x,y}$ for the Galois representation on the $l$-torsion of $E_{x,y}$
\[\rho_l^{x,y} : \Gal(\overline{\Q}/\Q) \rightarrow \Aut(E_{x,y}[l]).\]

\begin{lem}\label{lem:Irreducibility55l}
For primes $l\geq 7$ the representation $\rho_l^{x,y}$ is irreducible.
\end{lem}

\begin{proof}
Since $E_{x,y}$ has a rational $2$-isogeny, a reducible $\rho_l^{x,y}$ (for an odd prime $l$) would give rise to a noncuspidal $\Q$-rational point on the modular curve $X_0(2l)$. By work of Mazur et al. (see e.g. \cite[Section 2.1.2]{Dahmen}) this is impossible for primes $l\geq 11$ and only possible for $l=7$ if $j \in \{-3^3 \cdot 5^3,3^3 \cdot 5^3 \cdot 17^3\}$. Using our explicit formula for the $j$-invariant of $E_{x,y}$ we easily check that that there are no $[x:y] \in \mathbb{P}^1(\Q)$ giving rise to one of these two values for $j$.
\end{proof}

Now applying modularity and level-lowering we deduce the following.

\begin{lem}\label{lem:LevelLowering55l}
For primes $l \geq 7$ the Galois representation $\rho_l^{x,y}$ arises from a newform $f$ of level $N=2^\alpha 5$ where $\alpha \in \{1,3,4\}$.
\end{lem}

\begin{proof}
By \cite{BCDT01} we have that $\rho_l^{x,y}$ is modular of level $N(E_{x,y})$. Since by Lemma~\ref{lem:Irreducibility55l}  $\rho_l^{x,y}$ is also irreducible, we obtain by level lowering \cite{Ribet90}, \cite{Ribet94}, that $\rho_l^{x,y}$ is modular of level $N(E_{x,y})/M$ where $M$ is the product of all primes $p||N(E_{x,y})$ with $l \mid \ord_p(\Delta_{\mathrm{min}}(E_{x,y}))$. The possible values for $N(E_{x,y})$ and $\Delta_{\mathrm{min}}(E_{x,y})$ can be read off from Lemma~\ref{lem:DiscAndN55l}, which finishes the proof.
\end{proof}

We used {\tt MAGMA} to compute the newforms at these levels;
{\tt MAGMA} uses Stein's algorithms for this \cite{Stein}. We found
respectively $0$, $1$ and $2$ newforms at these levels, which are all rational. The (strong Weil) elliptic curves $E_0$ corresponding 
to these newforms are $E_{40a1}$, $E_{80a1}$, and $E_{80b1}$, where the subscript denotes the Cremona reference.
We wrote a short {\tt MAGMA} script {\tt Modular55l.m} which contains these, as well as the remaining computations of this section.
Comparing traces of Frobenius as usual, gives the following. 

\begin{lem}\label{lem:Traces55l}
Suppose that $\rho_l^{x,y} \simeq \rho_l^{E_0}$ for some prime $l \geq 7$ and some $E_0$ as above.
Let $p\not=2,5$ be a prime.
\begin{itemize}
\item If $p \nmid z$, then $a_p(E_0) \equiv a_p(E_{x,y}) \pmod{l}$.
\item If $p|z$, then $a_p(E_0) \equiv \pm (1+p) \pmod{l}$.
\end{itemize}
\end{lem}

\begin{proof}
See e.g. \cite[Propositions 15.2.2 and 15.2.3]{Cohen} or \cite[Theorem 36]{Dahmen}.
\end{proof}

We will now finish our intended proof.

\begin{proof}[Proof of Proposition~\ref{prop:SecondCase55l}]
By Lemma~\ref{lem:LevelLowering55l} and the determination of newforms of level $2^\alpha 5$ where $\alpha \in \{1,3,4\}$, we know that $\rho_l^{x,y}\simeq \rho_l^{E_0}$ for some prime $l \geq 7$ and $E_0$ one of $E_{40a1}$, $E_{80a1}$, $E_{80b1}$. We will eliminate these three possibilities for $E_0$, which then proves the proposition. Let $p\not=2,5$ denote a prime and define the sets
\[A_p:=\{p+1-\#{E}_{a,b}(\F_p) : a,b \in \F_p, \quad a^5+b^5 \not=0\}, \qquad T_p:=A_p \cup \{\pm (1+p)\}.\]
Obviously, if $p \nmid z$, then $a_p(E_{x,y}) \in A_p$. Hence by Lemma~\ref{lem:Traces55l} we have
\begin{equation}\label{eqn:TraceCongruence55l}
a_p(E_0) \equiv t \pmod{l} \mathrm{\ for\ some\ } t \in T_p.
\end{equation}
We compute
\[
T_3=\{\pm 2, \pm 4\}. 
\]
However, $E_{40a1}$ and $E_{80a1}$ have full $2$-torsion, and so $a_3(E_{40a1})=a_3(E_{80a1})=0$.
Thus for $l \geq 7$ prime and $E_0$ one of $E_{40a1}$ or $E_{80a1}$, we have that ~\eqref{eqn:TraceCongruence55l} with $p=3$ does not hold, and consequently $\rho_l^{x,y} \not\simeq E_0$. To deal with the remaining case $E_0=E_{80b1}$, we compute 
\[T_{43}=\{-44, -10, -8, -6, -2, 0, 2, 4, 6, 8, 12, 44\}, \quad a_{43}(E_{80b1})=10.\]
Now~\eqref{eqn:TraceCongruence55l} does not hold for any prime $l \geq 7$, \emph{except} $l=17$. So from now on let $l=17$, we deal with this case using the method of Kraus. For a prime $p\equiv 1 \pmod{l}$, let $\left(\F_p^*\right)^l$ denote the nonzero $l$-th powers in $\F_p$ and define the sets
\begin{align*}
A_{p,l}' & := \left\{p+1-\#{E}_{a,b}(\F_p) : a,b \in \F_p,\ 5(a+b) \in \left(\F_p^*\right)^l,\ H_5(a,b)/5 \in \left(\F_p^*\right)^l\right\},\\
T_{p,l}' & := A_{p,l}' \cup \{\pm 2\}.
\end{align*}
Now take $p=6\cdot 17+1=103$. Since we are assuming $\rho_{17}^{x,y} \simeq \rho_{17}^{E_0}$ (with $E_0=E_{80b1}$),
 Lemma~\ref{lem:Traces55l} gives us that
\begin{equation}\label{eqn:ExplicitKraus}
a_{103}(E_{80b1}) \equiv t \pmod{17} \quad \text{for some $t \in T_{103,17}'$}.
\end{equation}
We compute
\[T_{103,17}'=\{-6,\pm 2\}, \quad a_{103}(E_{80b1})=-14\]
and conclude that~\eqref{eqn:ExplicitKraus} does not hold, which completes the proof.
\end{proof}

\subsection{Necessity of the modular approach}\label{sec:Discussion}
Proving the nonexistence of non-trivial primitive integer solutions to~\eqref{eqn:55l} with $5|z$ for some odd prime $l$ can be reduced to finding $\Q$-rational points on the hyperelliptic curve
\[D_l: Y^2=4X^l+5^{2l-5}.\]
Note that this curve has genus $(l-1)/2$ and that
\[D_l(\Q) \supset \{\infty\}.\]

\begin{lem}\label{lem:ReductionToCurves55lCaseII}
Let $l$ be an odd prime. If
\begin{equation}\label{eqn:RatPoints55lCaseII}
D_l(\Q)=\{\infty\}.
\end{equation}
then there are no non-trivial primitive integer solutions to~\eqref{eqn:55l} with $5 | z$.
\end{lem}

\begin{proof}
In this case Lemma~\ref{lem:gcd} gives us $\gcd(x+y,H_5)=5$ and $5^2 \nmid H_5$. Hence \eqref{eqn:55lfactors} yields
\[5(x+y)= z_1^l, \qquad H_5=5 z_2^l\]
where $z_1$, $z_2$ are coprime non-zero integers satisfying $z=z_1 z_2$. Using identity \eqref{eqn:identity}
we see that 
\[5(x^2+y^2)^2=20z_2^l+5^{-4} z_1^{4l}.\]
Multiplying both sides by $5^{2l-1}/z_1^{4l}$ gives
\[\left(\frac{5^l (x^2+y^2)}{ z_1^{2l}}\right)^2=4 \left(\frac{5^2 z_2}{z_1^4}\right)^l+5^{2l-5}.\]
Thus
\[P=\left(\frac{5^2 z_2}{z_1^4}, \frac{5^l (x^2+y^2)}{z_1^{2l}}  \right) \in D_l(\Q).\]
Since $z_1 \neq 0$, we have $P \neq \infty$. This proves the lemma.
\end{proof}

Upper bounds for $\rank \Jac(D_l)(\Q)$ are given by the $2$-Selmer ranks of $\Jac(D_l)/\Q$; see Table~\ref{table:RankBoundsD}. For $l=7$ and $l=13$ (and, assuming GRH, also for $l=17$) we conclude that $\rank \Jac(D_l)(\Q)=0$, so it is easy to determine $D_l(\Q)$ for these values of $l$. Since our focus is on $l=7,19$, we give the details for $l=7$.

\begin{rem}
Instead of using a $2$-descent on $\Jac(D_l)/\Q$, we can also apply
\cite{StollI}, \cite{StollII} to get an upper bound for $\rank \Jac(D_l)(\Q)$
using a $(1-\zeta_l)$-descent on $\Jac (D_l))/\Q(\zeta_l)$. It turns out that for
$l=11$ this gives the same upper bound for $\rank \Jac(D_l)(\Q)$ as given by
Table~\ref{table:RankBoundsD} (namely $3$). For $l=7,13,17,19$ however, the
upper bounds obtained from a $(1-\zeta_l)$-descent will be strictly larger than
the bounds given by Table~\ref{table:RankBoundsD} (but one does not need to
assume GRH).  \end{rem}

\begin{table}[h!]
\begin{minipage}{\linewidth}
\caption{Rank bounds for the Jacobian of $D_l$}
\begin{tabular}{c|c|c}
$l$ & $\dim_{\F_2}\Sel^{(2)}(\Q,\Jac (D_l))~\footnote{The $*$ indicates that the result is conditional on GRH}$ & Time \\
\hline
7 & 0 & 1.4s \\
11 & 3 & 2093s \\
13 & 0 & 264613s $\approx$ 3.1 days \\
17 & $0^*$ & 240s \\
19 & $1^*$ & 723s \\
\end{tabular}
\label{table:RankBoundsD}
\end{minipage}
\end{table}

\begin{lem}\label{lem:RatPoints55lCaseII}
The only $\Q$-rational point on $D_7$ is the single point at infinity. 
\end{lem}

\begin{proof}
Let $J$ denote the Jacobian of $D_7$.
We shall show that $J(\Q)=\{0\}$. Since the Abel-Jacobi map
\[D_7 \rightarrow J, \qquad P \mapsto [P-\infty]\]
is injective, it will follow that $D_7(\Q)=\{\infty\}$.

First we determine the torsion subgroup $J(\Q)_\mathrm{tors}$ of $J(\Q)$. The curve $D_7$,
and hence its Jacobian $J$, has good reduction away from $2$, $5$, $7$.
For any (necessarily odd) prime $p$ of good reduction, the natural map
\[J(\Q)_\mathrm{tors} \rightarrow J(\F_p)\]
is injective. Using {\tt MAGMA} we find that
\[\# J(\F_3)=28, \qquad \# J(\F_{43})=39929.\]
Since $\gcd(28,39929)=1$, we deduce that $J(\Q)_\mathrm{tors}=\{0\}$.

We have already seen that $\rank J(\Q)=0$. It follows that
$J(\Q)=\{0\}$, which completes the proof.
\end{proof}

Let $r:=\rank \Jac(D_{19})(\Q)$. We see from Table~\ref{table:RankBoundsD} that $r \leq 1$
under the assumption of GRH. 
Assuming the finiteness of $\Sha(\Q,\Jac(D_{19}))$ in addition to GRH 
leads us to $r=1$. So in order to
use the method of Chabauty-Coleman to determine $D_{19}(\Q)$, we must first of
all prove that $r=1$ (if true \ldots) and next find a $\Q$-rational point of
infinite order on $\Jac(D_{19})$. Both tasks seem quite challenging at the moment.

We conclude that the modular method is not necessary to prove Theorem~\ref{thm:55l} for the case $l=7$, but that for $l=19$ we really do need it at this point.

\section{Proof of Theorem~\ref{thm:775}}\label{sec:77l}

In this section we shall be concerned with the primitive integer solutions to~\eqref{eqn:77l} for primes $l\not=2,3,7$. Although ultimately we will only to be able to (unconditionally) determine all the solutions if $l=5$, we will take a more general approach.
The reason for doing this is threefold. First of all, it is simply not much more work to consider more values of $l$.
Second, while we do not fully determine  (unconditionally) all primitive integer solutions to~\eqref{eqn:77l}  for any prime $l \geq 11$, we do obtain many other partial results for $l \geq 11$, which may be interesting in their own right. Finally, in Section~\ref{sec:GRH} we solve~\eqref{eqn:77l} for $l=11$ assuming GRH, for which we lay the foundations here.

\subsection{Initial factorizations for $x^7+y^7=z^l$}\label{sec:InitialFactorization77l}

Let $(x,y,z)$ be a primitive integer solution to \eqref{eqn:77l} for some prime $l\not=2,3,7$ and suppose that
$z \neq 0$. Recall that
\[
H_7(x,y)=\frac{x^7+y^7}{x+y}=
x^6-x^5 y+x^4 y^2 - x^3 y^3+x^2 y^4-x y^5+y^6.
\]
By Lemma~\ref{lem:gcd}, $\gcd(x+y,H_7(x,y))=1$ or $7$ and $7^2 \nmid H_7(x,y)$.
Thus we can again subdivide into two cases:
\begin{itemize}
\item If $7 \nmid z$ then
\begin{equation}\label{eqn:Factor77lCaseI}
x+y=z_1^l, \qquad H_7(x,y)= z_2^l, \qquad z=z_1 z_2
\end{equation}
where $z_1$, $z_2$ are non-zero, coprime integers. 
\item If 
$7 | z$ then
\begin{equation}\label{eqn:Factor77lCaseII}
7(x+y)=z_1^l, \qquad H_7(x,y)=7 z_2^l, \qquad z=z_1 z_2
\end{equation}
where $z_1$, $z_2$ are non-zero, coprime integers. 
\end{itemize}
These factorizations do not seem to be sufficient
to enable us to solve the problem. Henceforth, $\zeta$ will denote
a primitive $7$-th root of unity, $L=\Q(\zeta)$ and $\OO=\Z[\zeta]$
the ring of integers of $L$.
The class number of $\OO$
is $1$ and the unit rank is $2$. The unit group is in fact
\[\{ \pm \zeta^i (1+\zeta)^r (1+\zeta^2)^s : 0 \leq i \leq 6, \quad r,s \in \Z\}.\]
Moreover, $7$ ramifies as $7\OO=(1-\zeta)^6 \OO$.
Now $H_7(x,y)=\Norm(x+\zeta y)$. From~\eqref{eqn:Factor77lCaseI} and~\eqref{eqn:Factor77lCaseII}
we have
\begin{itemize}
\item If $7 \nmid z$ then
\begin{equation}\label{eqn:zetI}
x+\zeta y = (1+\zeta)^r (1+\zeta^2)^s \beta^l, \qquad 0 \leq r,s \leq l-1, 
\end{equation}
for some $\beta \in \Z[\zeta]$.
\item If $7|z$ then
\begin{equation}\label{eqn:zetII}
x+\zeta y = (1-\zeta) (1+\zeta)^r (1+\zeta^2)^s \beta^l, \qquad 0 \leq r,s \leq l-1, 
\end{equation}
for some $\beta \in \Z[\zeta]$.
\end{itemize}
Thus we have $2 l^2 \geq 50$ cases to consider. In the next section we will use the 
modular approach to reduce the number of cases to just $2$ for many values of $l$, e.g. $l=5,11$.

\subsection{A modular approach to $x^7+y^7=z^l$}\label{sec:Modular77l}

Consider the subset of primes
\[L_7:=\{\text{primes } l : l\not=2,3,7 \text{ and } l<100\}.\]
The purpose of this section is to prove the following proposition.

\begin{prop}\label{prop:Modular77l}
Let $(x,y,z)$ be a primitive integer solution to \eqref{eqn:77l} with
$z \neq 0$ and $l \in L_7$. 
If $7 \nmid z$ then \eqref{eqn:zetI} holds with $r=s=0$.
If $7 | z$ then \eqref{eqn:zetII} holds with $r=s=0$.
\end{prop}

Let $(x,y,z)$ be a primitive integer solution to~\eqref{eqn:77l} with $z\not=0$ for some prime $l \geq 5$, $l\not=7$.
Kraus \cite[pp. 329--330]{Kraus99} constructed a Frey curve for
the equation $x^7+y^7=z^l$. Following Kraus,
we associate to our solution $(x,y,z)$ to \eqref{eqn:77l}
the Frey elliptic curve
\[
E_{x,y} \; : \; Y^2= X^3+a_2 X^2 + a_4 X+ a_6,
\]
where
\begin{align*}
a_2 &=-(x-y)^2, \qquad
a_4=-2 x^4+ x^3 y -5 x^2 y^2 + x y^3-2y^4\\
a_6 &= x^6-6 x^5 y + 8 x^4 y^2 -13 x^3 y^3 + 8 x^2 y^4 -6 x y^5+y^6.
\end{align*}
We record some of the invariants of $E_{x,y}$:
\begin{align}
c_4 &=2^4 \cdot 7 (x^4 - x^3 y + 3 x^2 y^2 - x y^3 + y^4), \\
c_6 &=-2^5 \cdot 7 (x^6 - 15 x^5 y + 15 x^4 y^2 - 29 x^3 x^3 + 15 x^2 y^4 - 15 x y^5 + y^6),\\
\label{eqn:jinv} 
\Delta &=2^4 \cdot 7^2 H_7(x,y)^2, \qquad
j = \frac{2^8 \cdot 7 (x^4 - x^3 y + 3 x^2 y^2 - x y^3 + y^4)^3}{H_7(x,y)^2}.
\end{align}

\begin{lem}\label{lem:DiscAndN77l}
The conductor $N$ and minimal discriminant $\Delta_{\mathrm{min}}$ of $E_{x,y}$ satisfy
\begin{itemize}
\item $N=2^\alpha 7^2 \Rad(z_2)$ where $\alpha=2$ or $3$ 
and $\Rad(z_2)$
is the product of the distinct primes dividing $z_2$ (and $2,7 \nmid z_2$);
\item  $\Delta_{\mathrm{min}}=\Delta$.
\end{itemize}
\end{lem}

\begin{proof} 
Recall that $x$, $y$ are coprime. The resultant of $H_7(x,y)$ and $x^4 - x^3 y + 3 x^2 y^2 - x y^3 + y^4$
is $7^2$. Thus any prime $p \ne 2$, $7$ dividing $H_7(x,y)$ cannot divide $c_4$ and divides $\Delta$,
and must therefore be a prime of multiplicative reduction. We know that $H_7(x,y)=7 z_2^l$ or $H_7(x,y)=z_2^l$.
Moreover, $7^2 \nmid H_7(x,y)$, so $7 \nmid z_2$. Thus the conductor $N$ is $\Rad(z_2)$ up to powers of $2$ and $7$.
We also see that the model for $E_{x,y}$ is minimal at any prime $p \ne 2$, $7$

Now $\ord_7(\Delta)=4$ or $2$. Hence the model for $E_{x,y}$ is minimal at $7$.
Since $7 \mid c_4$, we see that $E_{x,y}$ has additive reduction
at $7$, and so $\ord_7(N)=2$.

Finally, as $x$, $y$ are coprime we quickly get $\ord_2(c_4)=4$, $\ord_2(c_6)=5$ as well as $\ord_2(\Delta)=4$.
Thus the model for $E_{x,y}$ is also minimal at $2$ and we conclude that $\Delta_{\mathrm{min}}=\Delta$.
Note that $E_{x,y}=E_{y,x}$. Without loss of generality we may suppose that either $x$
is even or $z$ is even.
Applying Tate's Algorithm \cite[Section IV.9]{SilvII} shows the following
\begin{enumerate}
\item[(a)] if $2 \mid z$ then $\ord_2(N)=3$;
\item[(b)] if $2 \mid\mid x$ then $\ord_2(N)=3$;
\item[(c)] if $4 \mid x$ then $\ord_2(N)=2$.
\end{enumerate}
This completes the proof.
\end{proof}

We shall write $\rho_l^{x,y}$ for the Galois representation on the $l$-torsion of $E_{x,y}$.
\[\rho_l^{x,y} : \Gal(\overline{\Q}/\Q) \rightarrow \Aut(E_{x,y}[l]).\]

\begin{lem}\label{lem:Irreducibility77l}
For $l=5$ or primes $l \geq 11$ the representation $\rho_l^{x,y}$ is irreducible.
\end{lem}

\begin{proof}
If $l=11$ or $l \geq 17$, then, by work of Mazur et al. on the $\Q$-rational points of $X_0(l)$, the irreducibility  follows by checking that the $j$-invariant of $E_{x,y}$ does not belong to an explicit list of $11$ values; see e.g.  \cite[Theorem 22]{Dahmen}.

Now let $l \in \{5,13\}$ and suppose that $\rho_l^{x,y}$ is reducible. Then the $j$-invariant of $E_{x,y}$
must be in the image of $X_0(l)(\Q)$ under the $j$ map $X_0(l) \rightarrow X(1)$. In \cite[Section 3.2]{Dahmen} this $j$ map is given explicitly
as
\begin{equation}\label{eqn:j}
j=
\begin{cases}
\frac{(t^2+10 t+ 5)^3}{t} & \text{if } l=5;\\
\frac{(t^4+7t^3+20t^2+19t+1)^3(t^2+5t+13)}{t} & \text{if } l=13.
\end{cases}
\end{equation}
In other words, this equation must have a $\Q$-rational solution $t$ where $j$
is the $j$-invariant of $E_{x,y}$. It is clear from \eqref{eqn:jinv}
that $\ord_2(j)=8$. It is easy to see that this is impossible from \eqref{eqn:j}. This completes the proof.
Alternatively, the irreducibility for $l \in \{5,13\}$ follows immediately from \cite[Theorem 60 and Table 3.1]{Dahmen} with $F(u,v)=u^3-u^2v-2uv^2+v^3$ and the remark that $F(x^2+y^2,xy)=H_7(x,y)$.
\end{proof}

Using Lemmata~\ref{lem:DiscAndN77l} and~\ref{lem:Irreducibility77l} we can apply modularity \cite{BCDT01} and level-lowering \cite{Ribet90}, \cite{Ribet94} as usual, to deduce the following.

\begin{lem}\label{lem:LevelLowering77l}
For a prime $l\not=2,3,7$, the Galois representation $\rho_l^{x,y}$ arises from a newform $f$ of level $N=2^\alpha 7^2$ where $\alpha=2$ or $3$.
\end{lem}

We again used {\tt MAGMA} to compute the newforms at these levels. We found
respectively $3$ and $8$ 
newforms (up to Galois conjugacy) at these levels.
Of these $2$ and $6$
are respectively rational newforms and therefore 
correspond to elliptic curves. We wrote a short {\tt MAGMA} script {\tt Modular77l.m} which contains these, as well as the remaining computations of this section.
Our first step is to eliminate as many of the newforms above as possible.

\begin{lem}\label{lem:Traces77l}
Suppose $\rho_l^{x,y}$ arises from a newform 
\begin{equation}\label{eqn:fourier}
f=q+\sum_{i \geq 2} a_i(f) q^i.
\end{equation}
Let $K=\Q(a_2(f),a_3(f),\dots)$ be the number field generated by the coefficients
of $f$. Let $p \ne 2$, $7$ be prime. If $K \ne \Q$ we also impose $p \ne l$.
\begin{itemize}
\item If $p \nmid z_2$, then $l \mid \Norm_{K/\Q} (a_p(E_{x,y})-a_p(f))$.
\item If $p \mid z_2$, then $l \mid \Norm_{K/\Q} ((p+1)^2-a_p(f)^2)$.
\end{itemize}
\end{lem}

\begin{proof}
This follows from comparing traces of Frobenius; see e.g. \cite[Propositions 15.2.2 and 15.2.3]{Cohen} or \cite[Theorem 36]{Dahmen}.
\end{proof}

Specializing $E_{x,y}$ at a trivial primitive integer solution with $xy=0$
(i.e. $(x,y)=(\pm 1, 0)$ or $(0,\pm 1)$ ), yields $E_{196a1}$, and specializing at
a trivial primitive integer solution with $z=0$ (i.e. $(x,y)=(\pm 1,\mp 1)$)
yields $E_{392c1}$. 
Using the basic congruences from the lemma above, we can quickly eliminate all the 
(Galois conjugacy classes of) newforms at the levels $196$ and $392$
for all $l$ simultaneously,  except of course the two
newforms corresponding to the two elliptic curves we just obtained by specialization of $E_{x,y}$.

\begin{lem}\label{lem:ResultAfterCongruences77l}
For a prime $l\not=2,3,7$, the Galois representation $\rho_l^{x,y}$ arises from $E_{196a1}$
or $E_{392c1}$. 
\end{lem}

\begin{proof}
By Lemma~\ref{lem:LevelLowering77l} we have that $\rho_l^{x,y}$ arises from a newform $f$ of level $2^\alpha 7^2$ where $\alpha=2$ or $3$. Let $p\not=2,7$ denote  a prime and define the sets
\begin{align*}
A_p & :=\{p+1-\#{E}_{a,b}(\F_p) : a,b \in \F_p, \quad H_7(a,b) \not=0\}, \\
T_p & :=
\begin{cases}
A_p & \text{if } p \not \equiv 1 \pmod{7}\\
A_p \cup \{\pm (1+p)\} & \text{if } p \equiv 1 \pmod{7}.
\end{cases}
\end{align*}
Obviously, if $p \nmid z_2$, then $a_p(E_{x,y}) \in A_p$. Furthermore, $p\equiv 1 \pmod{7}$ if and only $p$ splits completely in $\Z[\zeta]$ if and only $H_7(a,b)=0$ for some $a,b \in \F_p$ not both zero (for this last step we use $p\not=7$).
So we obtain from Lemma~\ref{lem:Traces77l} that for any prime $p\not=2,7$ we have
\begin{equation}\label{eqn:TraceCongruence77l}
l \mid \Norm_{K/\Q}(a_p(f) - t) \mathrm{\ for\ some\ } t \in T_p
\end{equation}
or, in case $K\not=\Q$, that $l=p$.

If $f$ is not rational, we compute that $a_3(f) \in \{\pm \sqrt{2}, \pm \sqrt{8}\}$ and $T_3=\{-1,3\}$. In this case~\eqref{eqn:TraceCongruence77l} with $p=3$ reduces to $l=7$, hence $l=7$ or $l=p=3$. Since $l=3,7$ are values outside our consideration we conclude that we have eliminated the possibility that $\rho_l^{x,y}$ arises form a non-rational newform. 
Similarly, for any rational newform $f$ (of level $2^\alpha 7^2$ where $\alpha
\in \{2,3\}$) not corresponding to either of $E_{196a1}$, $E_{392c1}$, 
we can find a single prime $p\leq 23$, $p\not=2,7$ such
that~\eqref{eqn:TraceCongruence77l} does not hold for any prime $l\not=2,3,7$.
To be specific, for the rational newforms corresponding to an elliptic curve
whose isogeny class has Cremona reference one of $196b, 392a, 392b, 392f$
we can take $p=3$, for the isogeny classes given
by $392e,392d$ we can take $p=11,13$ respectively. 
\end{proof}

So far we have not distinguished between the cases $7 \nmid z$ and $7|z$. To refine the lemma above with respect to these two cases we can use the following.

\begin{lem}\label{lem:ImagesOfInertia}
Let $E_1, E_2$ be elliptic curves over $\Q$ with potentially good reduction at a prime $p\geq 5$. If $\gcd(12,\ord_p(\Delta(E_1))) \not=\gcd(12,\ord_p(\Delta(E_2)))$, then for all primes $l\not=2,p$ we have $\rho_l^{E_1} \not\simeq \rho_l^{E_2}$.
\end{lem}

\begin{proof}
This follows by comparing images of inertia; see e.g. \cite{Kraus90}.
\end{proof}

We can now strengthen Lemma~\ref{lem:ResultAfterCongruences77l} as follows.

\begin{lem}\label{lem:ResultAfterImagesOfInertia77l}
Let $l\not=2,3,7$ be prime.
If $7 \nmid z$ then $\rho_l^{x,y}$ arises from  $E_{196a1}$.
If $7 |  z$ then $\rho_l^{x,y}$ arises from $E_{392c1}$.
\end{lem}

\begin{proof}
Considering $F:=x^4-x^3y+3x^2y^2-xy^3+y^4$ modulo $7$, we obtain that $7 \mid
F$ if and only if $7 \mid H_7$. Since $7^2 \nmid H_7$ we get from the
invariants of $E_{x,y}$ that $\ord_7(j) \geq 1$, so $E_{x,y}$ has potentially
good reduction at $7$. Furthermore, if $7 \nmid z$, then $\ord_7(\Delta)=2$,
and if $7 \mid z$, then $\ord_7(\Delta)=4$. The curves $E_{196a1}$ and $E_{392c1}$,
also have potentially good reduction at $7$ and finally
$\ord_7(\Delta(E_{196a1}))=2$
and
$\ord_7(\Delta(E_{392c1}))=4$.
The lemma follows from Lemma~\ref{lem:ImagesOfInertia}.
\end{proof}

\begin{rem}
To prove Lemma~\ref{lem:ResultAfterImagesOfInertia77l} we used image of inertia arguments.
It turns out that one can also eliminate $E_{196a1}$  when $7 \mid z$ for, say, $l<100$ with a simple application of Kraus' method.
The curve $E_{392c1}$ (when 7 $\nmid z$) is not susceptible to this method.
\end{rem}

We now turn our attention to a result involving the exponents $(r,s)$
in~\eqref{eqn:zetI} and~\eqref{eqn:zetII}, after which we will complete the proof
of Proposition~\ref{prop:Modular77l}.

\begin{lem}\label{lem:ExpSieve77l}
Let $E_0/\Q$ be an elliptic curve, let $p \ne 2$, $7$ be prime, let $l\not=2,3,7$ be prime, and let $g \in \{1,7\}$.
Denote by $\cA_g (E_0,p)$ the set of $(a,b) \in \F_p^2-\{0,0\}$ such that $(a+b)g$ and $H_7(a,b)/g$ are both $l$-th powers in $\F_p$, and
\begin{itemize}
\item either $H_7(a,b) \ne 0$ and $a_p(E_0) \equiv a_p(E_{a,b}) \pmod{l}$,
\item or $H_7(a,b)=0$ and $a_p(E_0)^2 \equiv (p+1)^2 \pmod{l}$.
\end{itemize}
Let $\p_1,\dots,\p_m$ be the prime ideals
of $\Z[\zeta]$ dividing $p$. Write $\kappa_i$ for the residue class field
$\Z[\zeta]/\p_i$ and $\pi_i$ for the corresponding natural map 
\[\pi_i : \Z[\zeta]/p \Z[\zeta] \rightarrow \kappa_i.\]
Denote by $\cB_g(E_0,p)$ the set of pairs $(\mu,\eta)$ with $0 \leq \mu,\eta < l$,
such that there exists $(a,b) \in \cA_g(E_0,p)$ with
\[\pi_i \left( \frac{a+b\zeta}{(1-\zeta)^{\ord_7(g)}(1+\zeta)^\mu (1+\zeta^2)^\eta } \right)\]
an $l$-th power in $\kappa_i$ for $i=1,\dots,m$.
\begin{enumerate}
\item[(a)]
If $E_{x,y}$ arises from $E_0$ and $7 \nmid z$, then~\eqref{eqn:zetI} holds for some $(r,s) \in \cB_1(E_0,p)$.
\item[(b)]
If $E_{x,y}$ arises from $E_0$ and $7 \mid z$, then~\eqref{eqn:zetII} holds for some $(r,s) \in \cB_7(E_0,p)$.
\end{enumerate}
\end{lem}

\begin{proof}
Let $g:=\gcd(x+y,H_7(x,y))$. By Lemma~\ref{lem:Traces77l} and~\eqref{eqn:Factor77lCaseI} and~\eqref{eqn:Factor77lCaseII}
we see that if $\rho_l^{x,y}$ arises from $E_0$, then $(x,y) \equiv (a,b) \pmod{l}$ for some $(a,b) \in \cA_g(E_0,p)$. 
The statement now follows directly by taking into account that the factorization of $x^7+y^7$ in $\Z[\zeta]$ yields~\eqref{eqn:zetI} and ~\eqref{eqn:zetII}.
\end{proof}

\begin{proof}[Proof of Proposition~\ref{prop:Modular77l}]
Let $(x,y,z)$ be a primitive integer solution to~\eqref{eqn:77l} with $l \in L_7$. 
We know that for some $0\leq r,s < l$ we have~\eqref{eqn:zetI} if $7 \nmid z$ and~\eqref{eqn:zetII} if $7 \mid z$.
Moreover, from Lemma~\ref{lem:ResultAfterImagesOfInertia77l} 
we know that $\rho_l^{x,y}$ arises from $E_{196a1}$ if $7 \nmid z$ and from $E_{392c1}$ if $7  \mid z$.
By Lemma~\ref{lem:ExpSieve77l},  for any prime $p \ne 2$, $7$, if $7 \nmid z$, then
\[(r,s) \in \cB_1(E_{196a1},p)\]
and if $7 \mid z$, then
\[(r,s) \in \cB_7(E_{392c1},p).\]
We wrote a short {\tt MAGMA} script to compute $\cB_g(E_0,p)$.
We found that for every prime $l \in L_7$ there exist primes $p_1, p_2$ such that
\[\cB_1(E_{196a1},p_1)=(0,0) \quad \text{and} \quad  \cB_7(E_{392c1},p_2)=(0,0).\]
This proves the proposition (see the  {\tt MAGMA} script {\tt Modular77l.m} for more details).
\end{proof}

\subsection{The hyperelliptic curves}

Assume $l \in L_7$ and let $(x,y,z)$ be a primitive integer solution to $x^7+y^7=z^l$ with $z \ne 0$.
Then according to Proposition~\ref{prop:Modular77l} we have
\begin{equation}\label{eqn:prenorm}
x+\zeta y= \epsilon \beta^l, \quad d (x+y)=z_1^l
\end{equation}
where $\beta \in \Z[\zeta]$ and 
\[(d,\epsilon) = 
\begin{cases}
(1,1) & \text{if } 7 \nmid z\\
(7, 1-\zeta) & \text{if } 7 | z.
\end{cases}\]
Let $\theta=\zeta+\zeta^{-1}$
and $K=\Q(\theta)$; this is the totally real cyclic cubic subfield of $L$. 
The Galois conjugates of $\theta$ are $\theta_1$, $\theta_2$,
$\theta_3$, which in terms of $\zeta$ are given by
\[\theta_1=\zeta+\zeta^{-1}, \qquad
\theta_2=\zeta^2+\zeta^{-2}, \qquad
\theta_3=\zeta^3+\zeta^{-3}.\]
Note that
\[\theta_1=\theta, \qquad
\theta_2=\theta^2-2, \qquad \theta_3=-\theta^2-\theta+1.\]
Let
\[\mu=\Norm_{L/K} (\epsilon), \qquad \gamma=\Norm_{L/K} (\beta).\]
Taking norms in \eqref{eqn:prenorm} down to $K$ we obtain
\begin{equation}\label{eqn:postnorm}
x^2+\theta x y+ y^2= \mu \gamma^l, \quad d (x+y)=z_1^l
\end{equation}
where $\gamma \in \mathcal{O}_K$ and 
\[(d, \mu) = 
\begin{cases}
(1,1) & \text{if } 7 \nmid z\\
(7, 2-\theta) & \text{if } 7 | z.
\end{cases}\]

Let $\mu_1=\mu$, $\mu_2$, $\mu_3$ denote the conjugates of $\mu$
that correspond respectively
to $\theta \mapsto \theta_j$, for $j=1,2,3$. Likewise
let $\gamma_1$, $\gamma_2$, $\gamma_3$ be the corresponding
conjugates of $\gamma$. Then
\[x^2+\theta_1 x y+ y^2= \mu_1 \gamma_1^l, \qquad
x^2+\theta_2 x y+ y^2= \mu_2 \gamma_2^l, \qquad
x^2+\theta_3 x y+ y^2= \mu_3 \gamma_3^l.\]
Furthermore, recall that
\[(x+y)^2=d^{-2} z_1^{2l}, \quad \text{where } d=
\begin{cases}
1 & \text{if } 7 \nmid z; \\
7 & \text{if } 7 | z.
\end{cases}\]
The left-hand sides of the previous four equations are symmetric binary quadratic forms over $K$. Since such forms obviously form a $2$-dimensional vector space over $K$ there exist linear relations between the four forms. We calculate
\begin{align*}
(x+y)^2+\theta_2 (x^2+\theta_1 x y+y^2)+\theta_3 (x^2+\theta_2 x y+ y^2)+\theta_1 (x^2+\theta_3 x y+ y^2) & = 0,\\
(x+y)^2+\theta_3 (x^2+\theta_1 x y+y^2)+\theta_1 (x^2+\theta_2 x y+ y^2)+\theta_2 (x^2+\theta_3 x y+ y^2) & = 0.
\end{align*}
This yields nice equations for a curve in projective $3$-space in the coordinates $z_1^2, \gamma_1, \gamma_2, \gamma_3$.
\begin{align*}
d^{-2} z_1^{2l}+\theta_2 \mu_1 \gamma_1^l+\theta_3 \mu_2 \gamma_2^l+\theta_1 \mu_3 \gamma_3^l & = 0,\\
d^{-2} z_1^{2l}+\theta_3 \mu_1 \gamma_1^l+\theta_1 \mu_2 \gamma_2^l+\theta_2 \mu_3 \gamma_3^l & = 0.
\end{align*}
We can eliminate one of the $\gamma_i$, say $\gamma_3$, to get
\begin{equation}\label{eqn:New77l}
(\theta_2-\theta_1) d^{-2} z_1^{2l}+(\theta_2^2-\theta_1 \theta_3) \mu_1 \gamma_1^l+(\theta_2 \theta_3-\theta_1^2) \mu_2 \gamma_2^l=0.
\end{equation}
And a projective plane curve in the coordinates $ \gamma_1, \gamma_2, \gamma_3$ is quickly obtained as
\begin{equation}\label{eqn:Samir77l}
(\theta_2-\theta_3) \mu_1\gamma_1^l+ (\theta_3-\theta_1) \mu_2 \gamma_2^l+(\theta_1-\theta_2) \mu_3 \gamma_3^l=0.
\end{equation}

\begin{rem}\label{rem:FermatQuotient}
Let $\alpha_1, \alpha_2, \alpha_3$ be nonzero elements in a field $F$ of characteristic $0$ and 
consider the nonsingular plane projective curve over $F$ determined by the equation
\begin{equation}\label{eqn:FermatGeneric}
\alpha_1 u^l+ \alpha_2 v^l+\alpha_3 w^l=0.
\end{equation}
Using the identity
\[(\alpha_1 u^l-\alpha_2 v^l)^2=(\alpha_1 u^l+\alpha_2 v^l)^2-4 \alpha_1 \alpha_2 (uv)^l,\]
we get from~\eqref{eqn:FermatGeneric} that
\[(\alpha_1 u^l-\alpha_2 v^l)^2=-4 \alpha_1 \alpha_2 (uv)^l+\alpha_3^2 w^{2l}.\]
By dividing both sides by $\alpha_3^2 w^{2l}$, we see that
\[\left(\frac{uv}{w^2},\frac{\alpha_1 u^l-\alpha_2 v^l}{\alpha_3 w^l}\right) \in C(F)\] 
where $C$ is the genus $(l-1)/2$ hyperelliptic curve determined by
\[C: Y^2= -4 \eta X^l+1, \quad \eta=\frac{\alpha_1 \alpha_2}{\alpha_3^2}.\]
Obviously, by permuting the indices, we find that $F$-rational points on~\eqref{eqn:FermatGeneric} also give rise to $F$-rational points on the hyperelliptic curves given by the equation above with $\eta=\alpha_2 \alpha_3/\alpha_1^2$ and $\eta=\alpha_3 \alpha_1/\alpha_2^2$ respectively.
\end{rem}

Define
\begin{gather*}
\alpha_1:=(\theta_2-\theta_1) d^{-2}, \quad
\alpha_2:=(\theta_2^2-\theta_1 \theta_3) \mu_1, \quad
\alpha_3:=(\theta_2 \theta_3-\theta_1^2) \mu_2;\\
\alpha_1' := (\theta_2-\theta_3) \mu_1,\qquad
\alpha_2' := (\theta_3-\theta_1) \mu_2,\qquad
\alpha_3' := (\theta_1-\theta_2) \mu_3;\\
\eta_1:=\alpha_2 \alpha_3/\alpha_1^2, \quad
\eta_2:=\alpha_3 \alpha_1/\alpha_2^2, \quad
\eta_3:=\alpha_1 \alpha_2/\alpha_3^2, \quad
\eta_4:=\alpha_1' \alpha_2'/\alpha_3'^2.
\end{gather*}
Then we see that Remark~\ref{rem:FermatQuotient} above leads to $K$-rational points on the curves 
\[Y^2=-4\eta_i X^l+1\]
for $i=1,2,3,4$ and the two possibilities for $(d,\mu)$. More precisely, if $7 \nmid z$, then
\begin{gather*}
\left(\frac{\gamma_1 \gamma_2}{z_1^4},\frac{\alpha_2 \gamma_1^l-\alpha_3 \gamma_2^l}{\alpha_1 z_1^{2l}} \right) \in C_{l,1}(K), \quad
\left(\frac{\gamma_2 z_1^2}{\gamma_1^2}, \frac{\alpha_3 \gamma_2^l-\alpha_1 z_1^{2l}}{\alpha_2 \gamma_1^l} \right) \in C_{l,2}(K), \\
\left(\frac{z_1^2 \gamma_1}{\gamma_2^2}, \frac{\alpha_1 z_1^{2l}-\alpha_2 \gamma_1^l}{\alpha_3 \gamma_2^l}  \right) \in C_{l,3}(K), \quad
\left(\frac{\gamma_1 \gamma_2}{\gamma_3^2}, \frac{\alpha_1' \gamma_1^l-\alpha_2' \gamma_2^l}{\alpha_3' \gamma_3^l} \right) \in C_{l,4}(K)
\end{gather*}
where $C_{l,i}$ denotes the genus $(l-1)/2$ hyperelliptic curve given by
\[C_{l,i}: Y^2=-4 \eta_i X^l+1, \quad (\mu,d)=(1,1), \quad i=1,2,3,4.\]
If $7 | z$, then similarly
\begin{gather*}
\left(\frac{\gamma_1 \gamma_2}{z_1^4},\frac{\alpha_2 \gamma_1^l-\alpha_3 \gamma_2^l}{\alpha_1 z_1^{2l}} \right) \in D_{l,1}(K), \quad
\left(\frac{\gamma_2 z_1^2}{\gamma_1^2}, \frac{\alpha_3 \gamma_2^l-\alpha_1 z_1^{2l}}{\alpha_2 \gamma_1^l} \right) \in D_{l,2}(K), \\
\left(\frac{z_1^2 \gamma_1}{\gamma_2^2}, \frac{\alpha_1 z_1^{2l}-\alpha_2 \gamma_1^l}{\alpha_3 \gamma_2^l}  \right) \in D_{l,3}(K), \quad
\left(\frac{\gamma_1 \gamma_2}{\gamma_3^2}, \frac{\alpha_1' \gamma_1^l-\alpha_2' \gamma_2^l}{\alpha_3' \gamma_3^l} \right) \in D_{l,4}(K)
\end{gather*}
where $D_{l,i}$ denotes the genus $(l-1)/2$ hyperelliptic curve given by
\[D_{l,i}: Y^2=-4 \eta_i X^l+1, \quad (\mu,d)=(2-\theta,7), \quad i=1,2,3,4.\]
The possible values of $\eta_i$ are given explicitly in Table~\ref{table:etaValuesC77l}.
Note that if $7|z$, then $\eta_2 = -\eta_3$, hence $D_{l,2} \simeq D_{l,3}$.

\begin{table}[h!]
\caption{Values of $\eta_i$}
\begin{scriptsize}
\begin{tabular}{c|c|c|c|c}
$(\mu,d)$ & $\eta_1$ & $\eta_2$ & $\eta_3$ & $\eta_4$  \\
\hline
$(1,1)$ & $2\theta^2 + \theta - 5$ & $-5\theta^2 + 4\theta + 3$ & $-\theta^2 - 3\theta - 2$ & $\theta^2 - 3$\\
$(2-\theta,7)$ & $7^4(20\theta^2 + 11\theta - 46)$ & $7^{-3}(-\theta^2 + 4\theta + 3)$ & $7^{-3}(\theta^2 - 4\theta - 3)$ & $20\theta^2 + 11\theta - 45$
\end{tabular}
\end{scriptsize}
\label{table:etaValuesC77l}
\end{table}

Next, we note that there must be a linear dependence between the symmetric binary quadratic forms $(x-y)^2$, $(x+y)^2$, and $x^2+\theta xy +y^2$. It is given by
\[(\theta-2)(x-y)^2=-4(x^2+\theta xy+y^2)+(\theta+2)(x+y)^2.\]
Using $(x+y)^2=d^{-2}z_1^{2l}$ and $x^2+\theta xy+y^2=\mu \gamma^l$, we get
\begin{equation}\label{eqn:SemiQ77l}
\left(\frac{x-y}{x+y}\right)^2=\frac{-4 \mu d^2}{\theta-2} \left(\frac{\gamma}{z_1^2}\right)^l+\frac{\theta+2}{\theta-2}.
\end{equation}
So if $7 \nmid z$, then
\begin{equation}\label{eqn:form}
\left(\frac{\gamma}{z_1^2},\frac{x-y}{x+y}\right) \in C_{l,0}(K)
\end{equation}
where $C_{l,0}$ denotes the genus $(l-1)/2)$ hyperelliptic curve given by
\[C_{l,0}: Y^2=7^{-1}(4\theta^2 + 12\theta + 16) X^l+7^{-1}(-4\theta^2 - 12\theta - 9).\]
If $7 | z$, then
\[\left(\frac{\gamma}{z_1^2},\frac{x-y}{x+y}\right) \in D_{l,0}(K)\]
where $D_{l,0}$ denotes the genus $(l-1)/2$ hyperelliptic curve given by
\[D_{l,0}: Y^2=14^2 X^l+7^{-1}(-4\theta^2 - 12\theta - 9).\]

Thus we have reduced our problem to determining the $K$-rational points on two genus $(l-1)/2$ curves. Namely one of the $C_{l,i}$ and one of the $D_{l,i}$. Note that
\begin{align}
\label{eqn:RatPoints77lCaseI}
C_{l,i}(K)  & \supset 
\begin{cases}
 \{\infty, (1,\pm 1) \} & \text{if } i=0\\
 \{\infty, (0,\pm 1), (1,\pm (2\theta^2 - 5)) \} & \text{if } i=1\\
 \{\infty, (0,\pm 1), (1,\pm (2\theta^2 - 2\theta - 1)) \} & \text{if } i=2\\
 \{\infty, (0,\pm 1), (1,\pm (2\theta + 3)) \} & \text{if } i=3\\
 \{\infty, (0,\pm 1), (1,\pm (2\theta^2 + 2\theta - 3)) \} & \text{if } i=4
 \end{cases}\\
 \label{eqn:RatPoints77lCaseII}
D_{l,i}(K) & \supset 
\begin{cases}
\{\infty \} & \text{if } i=0\\
\{\infty, (0,\pm 1) \} & \text{if } i=1,2,3\\
\{\infty, (0,\pm 1), (1,\pm (6\theta^2 + 4\theta - 13)) \} & \text{if } i=4.
\end{cases}
\end{align}

\begin{lem}\label{lem:ReductionToCurves77l}
Let $l \in L_7$. 
\begin{itemize}
\item If for at least one $i \in \{0,1,2,3,4\}$ equality holds in~\eqref{eqn:RatPoints77lCaseI}, then there are no non-trivial primitive integer solutions to $x^7+y^7=z^l$ with $7\nmid z$.
\item If for at least one $i \in \{0,1,2,3,4\}$ equality holds in~\eqref{eqn:RatPoints77lCaseII}, then there are no non-trivial primitive integer solutions to $x^7+y^7=z^l$ with $7 | z$.
\end{itemize}
\end{lem}

\begin{proof}
Let $(x,y,z)$ be a non-trivial primitive integer solution to~\eqref{eqn:77l}. We have seen that this gives rise to a $P=(X,Y) \in  C_{l,i}(K)$ for all $i\in \{0,1,2,3,4\}$ if $7 \nmid z$ and it gives rise to a $P=(X,Y) \in D_{l,i}(K)$ for all $i \in \{0,1,2,3,4\}$ if $7 | z$. Obviously, $P\not=\infty$ and $X\not=0$. So the first part of the lemma (i.e.\ the $7 \nmid z$ case) follows if we prove that $X \not=1$, and the second part of the lemma (i.e. the $7 | z$ case) follows if we prove that $X\not=1$ if $i=4$. Let $\gamma=\gamma_1, \gamma_2, \gamma_3, z_1$ be as before. Note that they are nonzero pairwise coprime algebraic integers in $K=\Q[\theta]$ and of course $z_1 \in \Z$. Also note that the
roots of unity in $\Z[\theta]$ are $\pm 1$. For $i=0,1,2,3,4$ we have respectively
\[X=\frac{\gamma}{z_1^2}, \frac{\gamma_1 \gamma_2}{z_1^4}, \frac{\gamma_2 z_1^2}{\gamma_1^2}, \frac{z_1^2 \gamma_1}{\gamma_2^2}, \frac{\gamma_1 \gamma_2}{\gamma_3^2}.\]
Furthermore, recall that
\[x^2+\theta xy+ y^2= \mu \gamma^l\]
where $\mu=1$ if $7 \nmid z$, and $\mu=2-\theta$ if $7 | z$.

Let us assume that $7 \nmid z$. 
From the condition $X=1$ we now see that $z_1^2=1$ and that the $\gamma_i$ are units. If $i=0$, then we get $\gamma=1$. If $i=1$, then we get $1=\gamma_1 \gamma_2=\Norm(\gamma)/\gamma_3=\pm 1/\gamma_3$, hence $\gamma=\pm 1$. If $i=2$, then $\gamma_2=\gamma_1^2$, and from the Galois action we see that $\gamma_2^8=\gamma_2$, which implies $\gamma_2=1$ and hence $\gamma=1$. If $i=3$, then similar as in the previous case we get to $\gamma=1$. Finally, if $i=4$, then $1=\gamma_1 \gamma_2/\gamma_3^2=\Norm(\gamma)/\gamma_3^3=\pm 1/\gamma_3^3$, which implies $\gamma_3=\pm 1$ and hence $\gamma=\pm 1$. In all cases we see that $\gamma=\pm 1$, so
\[x^2+\theta xy+y^2= \pm 1.\]
Since $x,y, \in \Z$ we get $xy=0$. A contradiction which proves the first part of the lemma.

Now assume $7 |z$. We let $i=4$. The condition $X=1$ implies, as before, that $\gamma= \pm 1$. This gives us
\[x^2+\theta xy+y^2= \pm (2-\theta).\]
The integer solution are $(x,y)=(\pm 1, \mp 1)$, hence $z=0$. A contradiction which proves the second part of the lemma.
\end{proof}

\begin{rem}
We know of at least one instance where equality does not hold in ~\eqref{eqn:RatPoints77lCaseII}, namely
\begin{equation}\label{eqn:RatPointsD13_1}
D_{13,1}(K) \supset \{\infty, (0,\pm 1) , (7^{-1}(3\theta^2 + 2\theta - 2) , \pm (4\theta^2 + 6\theta + 1))\} .
\end{equation}
It is of course a simple matter to check that the pair of \lq new\rq\ points does not come from a non-trivial primitive integer solution to~\eqref{eqn:77l}, from which we conclude that equality in~\eqref{eqn:RatPointsD13_1} implies the nonexistence of non-trivial primitive integer solutions to~\eqref{eqn:77l} with $7 | z$ and $l=13$. Although it seems very likely that indeed this equality holds, proving it still remains quite a challenge.
\end{rem}

\begin{rem}\label{rem:ExtraInfo}
Instead of finding the full set $S$ of $K$-rational points on one of the $C_{l,i}$ or $D_{l,i}$ in order to apply Lemma~\ref{lem:ReductionToCurves77l}, it can be convenient to use extra (local) information so that the same conclusion can be obtained by finding a specific subset of $S$ satisfying extra (local) conditions. For example, let $\p$ be the prime above $7$, then for $j=1,2,3$ we have $x^2+\theta_j x y+y^2 \equiv (x+y)^2 \pmod{\p}$. So for a primitive integer solution to~\eqref{eqn:77l} with $7 \nmid z$ we get, using $\gamma_1, \gamma_2, \gamma_3, z_1$ as before, that $\gamma_1^l \equiv \gamma_2^l \equiv \gamma_3^l \equiv (z_1^2)^l \pmod{\p}$. Since $l\not=2,3$ and $7 \nmid z$ we obtain respectively
\begin{equation}\label{eqn:LocalInfoAt7}
\gamma_1 \equiv \gamma_2 \equiv \gamma_3 \equiv z_1^2 \pmod{\p}, \quad \gamma_1 \gamma_2 \gamma_3 z_1^2 \not\equiv 0 \pmod{\p}.
\end{equation}
We note that $C_{l,i}$ for $i=1,2,3,4$ has good reduction at $\p$. Now the local information~\eqref{eqn:LocalInfoAt7} implies that our solution gives rise to a point $\tilde{P}_i$ on the reduction $\tilde{C}_{l,i}/\F_7$ where
\[\tilde{P}_i= (1,3), (1,4), (1,0), (1,2)\]
for $i=1,2,3,4$ respectively. Therefore define for $i=1,2,3,4$
\[C_{l,i}(K)':=\{ P \in C_{l,i}(K) : P \pmod{\p}=\tilde{P}_i\}.\]
For the curve $C_{l,0}$ we see, by \eqref{eqn:form},
that any $P \in C_{l,0}(K)$ that comes from a solution to \eqref{eqn:77l}
has second coordinate in $\Q$, where by convention we say that $\infty$ has second coordinate in $\Q$. Therefore define
\[C_{l,0}(K)':=\{ P \in C_{l,0}(K) : P \text{\ has second coordinate in } \Q\}.\]
We arrive at the following refined version of the first part of Lemma~\ref{lem:ReductionToCurves77l}.

\begin{lem}\label{lem:ReductionToCurves77lBis}
Let $l \in L_7$. If for at least one $i \in \{0,1,2,3,4\}$ we have
\[C_{l,i}(K)' = 
\begin{cases}
\{\infty, (1,\pm 1)\} & \text{if } i=0 \\
\{(1,2\theta^2 - 5)\} & \text{if } i=1\\
\{(1,-2\theta^2 + 2\theta + 1)\} & \text{if } i=2 \\
\{(1,\pm (2\theta + 3))\} & \text{if } i=3 \\
\{(1,2\theta^2 + 2\theta - 3)\} & \text{if } i=4,
\end{cases}\]
then there are no non-trivial primitive integer solutions to $x^7+y^7=z^l$ with $7\nmid z$.
\end{lem}

Similar remarks apply to $D_{l,0}$ and $D_{l,4}$.
\end{rem}

\subsection{Rational points on $C_{l,i}$ and $D_{l,i}$}\label{sec:Chabauty775}

The curves $C_{l,i}$ for $i=0,\ldots, 4$ and $D_{l,4}$ contain a $K$-rational point $P=(X,Y)$ with $X=1$. We can check that $D:=[P-\infty]$ is a point of infinite order on the Jacobian. Upper bounds for the ranks of the Jacobians of the $C_{5,i}$ and the $D_{5,i}$ can be found in Tables~\ref{table:RankBoundsC775} and \ref{table:RankBoundsD775} respectively. We conclude that 
\begin{gather*}
 \rank \Jac(C_{5,1})(K)=\rank \Jac(C_{5,2})(K)=\rank \Jac(C_{5,3})(K)=1\\
 \rank \Jac(D_{5,4})(K)=1, \quad \rank \Jac(D_{5,2})(K)=0.
 \end{gather*}

\begin{table}[h!]
\caption{Rank bounds for the Jacobian of $C_{5,i}$}
\begin{tabular}{c|c|c}
$C$ & $\dim_{\F_2}\Sel^{(2)}(K,\Jac (C))$ & Time  \\
\hline
$C_{5,0}$ & 2 & 1545s \\
$C_{5,1}$ & 1 & 1667s \\
$C_{5,2}$ & 1 & 1700s \\
$C_{5,3}$ & 1 & 1928s \\
$C_{5,4}$ & 2 & 571s \\
\end{tabular}
\label{table:RankBoundsC775}
\end{table}

\begin{table}[h!]
\caption{Rank bounds for the Jacobian of $D_{5,i}$}
\begin{tabular}{c|c|c}
$D$ & $\dim_{\F_2}\Sel^{(2)}(K,\Jac (D))$ & Time  \\
\hline
$D_{5,0}$ & 1 & 79083s $\approx$ 22.0h\\
$D_{5,1}$ & 1 & 89039s $\approx$ 24.7h \\
$D_{5,2}(\simeq D_{5,3})$ & 0 & 102817s $\approx$ 28.6h \\
$D_{5,4}$ & 1 & 1838s \\
\end{tabular}
\label{table:RankBoundsD775}
\end{table}

We see that we are in a good position to solve~\eqref{eqn:77l} for $l=5$. For the case $7 \nmid z$ the candidates $C_{5,1}, C_{5,2}$, and $C_{5,3}$ seem equally promising at this point, we choose to work with $C_{5,3}$. For the case $7 |z$, the curves $D_{5,1}$ and $D_{5,2}$ are both good candidates, but obviously $D_{5,2}$ is the easier one to work with, since its Jacobian has rank zero.

\begin{prop}\label{prop:RatPoints775}
We have
\begin{align*}
C_{5,3}(K)' & = \{(1,\pm (2\theta + 3))\}, \\
D_{5,2}(K) & = \{ \infty, (0,\pm 1) \}.
\end{align*}
\end{prop}

\begin{proof}
We will first determine $C_{5,3}(K)'$ and write for now $J:=\Jac(C_{5,3})$. Let $P_{\pm}:=(1,\pm(2\theta+3)) \in C_{5,3}(K)$ and $D:=[P_+-\infty] \in J(K)$. Then, as remarked before, $D$ has infinite order. Since we need this fact in the proof, we will supply details here. Using explicit computations in {\tt MAGMA} it is straightforward to check this, but it can actually easily be shown \lq by hand\rq\ as follows. Note that $C_{5,3}$ and hence $J$ have good reduction at the prime $\p$ above $7$, denote the reductions by $\tilde{C}_{5,3}$ and $\tilde{J}$ respectively. The points $P_{\pm}$ reduce to a single Weierstrass point $\tilde{P}=(1,0) \in \tilde{C}_{5,3}(\F_7)$. Thus the reduction $\tilde{D}$ of $D$ has order $2$ in $\tilde{J}(\F_7)$. Since the hyperelliptic polynomial $f:=-4\eta_3X^5+1$ in the defining equation for $C_{5,3}$ is irreducible, we get that $\#J(K)_{\text{tors}}$ is odd. This implies that any elements of $J(K)$ whose reduction modulo a prime of good reduction has even order cannot be torsion, in particular $D$ has infinite order. 

Now we will apply Chabauty-Coleman with the prime $\p$. A basis for $\Omega(C_{5,3}/K_{\p})$ is given by $X^i dX/Y$ with $i=0,1$.
We have explicitly $2D=[P_+-P_-]$, which also has infinite order of course. We note that the rational function $X-1$ does not reduce to a local uniformizer at $\tilde{P}$, but the function $T:=Y+Y_0$ does, where $Y_0:=2\theta+3$. We compute $2YdY=-20\eta_3 X^4 dX$, so
\[X^i\frac{dX}{Y}=X^i\frac{dY}{-10\eta_3 X^4}=\frac{dT}{-10\eta_3 X^{4-i}}.\]
Furthermore, (with the obvious choice for the $5$-th root) we have around $P_-$
\begin{gather*}
X^{-1}=\left(\frac{Y^2-1}{-4\eta_3}\right)^{-1/5}=\left(1+\frac{T^2-2Y_0T}{-4\eta_3}\right)^{-1/5}=\\
\quad 1+\frac{-\theta + 2}{10}T+\frac{13\theta^2 - 17\theta + 2}{100}T^2+\frac{287\theta^2 - 274\theta - 103}{1000}T^3+\ldots \in K[[T]].
\end{gather*}
Formal integration allows us to calculate to high-enough $\p$-adic precision
\[c_i:= \int_0^{2D} X^i\frac{dX}{Y}=\int_{P_-}^{P_+} X^i\frac{dX}{Y}=\int_0^{2Y_0} \frac{dT}{-10\eta_3 X^{4-i}}.\]
We note that $v_{\p}(c_0)=v_{\p}(c_1)=1$. Now $\omega:=(-c_1/c_0 +X)/dY \in \Ann(J(K))$ and the function 
\[f(T):=\int_0^{T} \frac{(-c_1/c_0 +X(T'))dT'}{-10\eta_3 X(T')^4}\]
vanishes for $T \in Y_0 \mathcal{O}_{K_{\p}}$ such that $(X(T),Y(T)) \in C_{5,3}(K)$, which have to reduce mod $\p$ to $\tilde{P}$.
The Strassmann bound for the power series in $t$ of $f(Y_0t)$ can be computed to be $3$. The zeroes $t=0$ and $t=2$ correspond to the points $P_-$ and $P_+$ respectively. The third solution occurs at $t=1$, which corresponds to the unique Hensel-lift of $\tilde{P}$ 
to a $\p$-adic Weierstrass point. This last point is not $K$-rational (since $f$ is irreducible over $K$), so we conclude that $C_{5,3}(K)' = \{P_{\pm}\}$. Further details can be found in our {\tt MAGMA} script {\tt Chabauty77l.m}.

Determining $D_{5,2}(K)$ is straightforward, since $J:=\Jac(D_{5,2})$ has rank $0$. The number of points on the reduction of $J$ at the prime above $p$ for $p=3, 11$ respectively can be calculated to equal $730$ and $1882705$ respectively. Their gcd equals $5$. Since $[(0,1)-\infty] \in J(K)$ is non-trivial, it must be a point of order $5$ generating $J(K)$. The Abel-Jacobi map 
\[D_{5,2}(K) \to J(K): \quad P \mapsto [P-\infty]\]
is injective. The points $n[(0,1)-\infty]$ for $n=2,3$ cannot be represented as $[P-\infty]$ for some $P \in D_{5,2}(K)$. This shows that
$D_{5,2}(K) = \{ \infty, (0,\pm 1) \}$.
\end{proof}

Obviously, the proposition above together with Lemmata~\ref{lem:ReductionToCurves77l} and~\ref{lem:ReductionToCurves77lBis} imply Theorem~\ref{thm:775}.

\begin{rem}
With a bit more work it is possible to determine $C_{5,3}(K)$ completely as well as $C_{5,1}(K), C_{5,2}(K)$, and $D_{5,4}(K)$. In an earlier version of this paper we only dealt with the curves $C_{5,4}$ and $D_{5,4}$, so we had to determine $C_{5,4}(K)$ as well. For this curve it is in fact possible to find another independent $K$-rational point on the Jacobian and use Chabauty over number fields \cite{Siksek} to determine $C_{5,4}(K)$ on this genus $2$ curve of rank $2$ over $K$.  
\end{rem}

\section{Results assuming GRH}\label{sec:GRH}

The purpose of this section is to prove Theorem~\ref{thm:GRH}.  There are however many other, unconditional, results in this section, which can be interesting in their own right. When a result is conditional on GRH, we shall clearly state so.
We shall start with the equation $x^7+y^7=z^l$, since the treatment is a direct continuation of the previous section. After this, the equation $x^5+y^5=z^l$ will be revisited. In the final section we shall briefly discuss the possibility of making the results unconditional.

\subsection{The equation $x^7+y^7=z^l$ for $l=11,13$}

As in the $l=5$ case, we can check that for $l \in \{11,13\}$ the $K$-rational points on $C_{l,i}$ for $i=0,1,2,3,4$ and $D_{l,4}$ give rise to a point of infinite order on their Jacobians. Assume GRH. Rank bounds for the Jacobians of the $C_{l,i}$ and the $D_{l,i}$ with $l \in \{11,13\}$ can be found in Tables~\ref{table:RankBoundsC77l} and~\ref{table:RankBoundsD77l} respectively. 
We want to stress again that because of the pseudo-random number generator involved in computing the ranks, the computation time also depends (really heavily this time) on the seed.
We conclude from the tables that
\[\rank \Jac(C_{11,3})(K)=1, \quad \rank \Jac(D_{11,4})(K)=1\]
and of course 
\[\rank \Jac(D_{11,0})(K)=0, \quad \rank \Jac(D_{13,2})(K)=0.\]

\begin{table}[h!]
\caption{GRH Rank bounds for the Jacobian of $C_{l,i}$}
\begin{tabular}{c|c|c}
$C$ & $\dim_{\F_2}\Sel^{(2)}(K,\Jac(C))$ & Time  \\
\hline
$C_{11,0}$ & 4 & 10481s $\approx$ 2.9h\\
$C_{11,1}$ & 3 & 4226s $\approx$ 1.2h\\
$C_{11,2}$ & 2 & 7207s $\approx$ 2.0h \\
$C_{11,3}$ & 1 & 3604s $\approx$ 1.0h\\
$C_{11,4}$ & 2 & 14816s $\approx$ 4.1h\\
\hline
$C_{13,0}$ & 2 & 10508s $\approx$ 2.9h\\
$C_{13,1}$ & 2 & 365096s $\approx$ 4.2 days\\
$C_{13,2}$ & 2 &  108629s $\approx$ 30.2h\\
$C_{13,3}$ & 4 & 107770s $\approx$ 29.9h\\
$C_{13,4}$ & 3 & 119062s $\approx$ 33.1h\\
\end{tabular}
\label{table:RankBoundsC77l}
\end{table}

\begin{table}[h!]
\caption{GRH Rank bounds for the Jacobian of $D_{l,i}$}
\begin{tabular}{c|c|c}
$D$ & $\dim_{\F_2}\Sel^{(2)}(K,\Jac(D))$ & Time  \\
\hline
$D_{11,0}$ & 0 & 6419s $\approx$ 1.8h\\
$D_{11,1}$ & 1 & 7550s $\approx$ 2.1h\\
$D_{11,2}(\simeq D_{11,3})$ & 2 & 12010s $\approx$ 3.3h\\
$D_{11,4}$ & 1 &  1800s $\approx$ 0.5h\\
\hline
$D_{13,0}$ & 2 & 469263s $\approx$ 5.4 days\\
$D_{13,1}$ & 3 & 91258s $\approx$ 25.3h\\
$D_{13,2}(\simeq D_{13,3})$ & 0 & 43182s $\approx$ 12.0h\\
$D_{13,4}$ & 3 &  10225s $\approx$ 2.8h\\
\end{tabular}
\label{table:RankBoundsD77l}
\end{table}

We see that we are in a good position to solve~\eqref{eqn:77l} for $l=11$, but that we have insufficient information to treat the $7 \nmid z$ case when $l=13$.

\begin{prop}
Assuming GRH, we have
\begin{align*}
C_{11,3}(K)' & = \{(1,\pm (2\theta + 3))\}, \\
D_{11,0}(K) & = \{ \infty\}.
\end{align*}
\end{prop}

\begin{proof}
The proof that $C_{11,3}(K)' = \{(1,\pm (2\theta + 3))\}$ is analogous to our proof that $C_{5,3}(K)' = \{(1,\pm (2\theta + 3))\}$ given in Proposition~\ref{prop:RatPoints775}. Details can be found in our {\tt MAGMA} script {\tt Chabauty77l.m}.

Since $\rank \Jac(D_{11,0})(K)=0$ we can get $D_{11,0}(K) = \{ \infty\}$ from the fact that $\Jac(D_{11,0})(K)_{\text{tors}}$ is trivial. This last statement follows from observing that the defining equation for $D_{11,0}$ shows that $\#\Jac(D_{11,0})(K)_{\text{tors}}$ is odd and counting points on the reduction of $\Jac(D_{11,0})$ modulo the prime above $5$ and a prime above $13$.
\end{proof}

\begin{cor}
Assuming GRH, there are no non-trivial primitive integer solutions to~\eqref{eqn:77l} for $l=11$.
\end{cor}

\subsection{The equation $x^5+y^5=z^l$ revisited}

Instead of just working over $\Q$, like we did in Section~\ref{sec:55l}, we shall use the factorization of $H_p$ over $\Q(\zeta_p)$ and $\Q(\zeta_p+\zeta_p^{-1})$, like we did in Section~\ref{sec:77l}, but now with $p=5$ instead of $p=7$ of course.

\subsubsection{Initial factorizations for $x^5+y^5=z^l$}\label{sec:InitialFactorization55l}

Let $(x,y,z)$ be a primitive integer solution to \eqref{eqn:55l} with $5 \nmid z$ for some prime $l>5$. Recall that
\[H_5(x,y)=\frac{x^5+y^5}{x+y}=
x^4 - x^3 y + x^2 y^2 - x y^3 + y^4.\]
By Lemma~\ref{lem:gcd}, $\gcd(x+y,H_5(x,y))=1$, and consequently
\[x+y=z_1^l, \qquad H_5(x,y)= z_2^l, \qquad z=z_1 z_2\]
where $z_1$, $z_2$ are non-zero, coprime integers. 

Let $\zeta$ denote a primitive $5$-th root of unity, $L=\Q(\zeta)$ and $\OO=\Z[\zeta]$
the ring of integers of $L$.
The class number of $\OO$
is $1$ and the unit rank is $1$. The unit group is in fact
\[
\{ \pm \zeta^i (1+\zeta)^r  : 0 \leq i \leq 4, \quad r \in \Z\}.
\]
Moreover, $5$ ramifies as $5\OO=(1-\zeta)^4 \OO$.
Now $H_5(x,y)=\Norm(x+\zeta y)$. We have
\begin{equation}\label{eqn:ZetaFactor55l}
x+\zeta y = (1+\zeta)^r  \beta^l, \qquad 0 \leq r \leq l-1, 
\end{equation}
for some $\beta \in \Z[\zeta]$.
Thus we have $l$ cases to consider. Using a modular approach, we can reduce the number of cases to just $1$ for many values of $l$, e.g. $l=11,13,17$.

\subsubsection{A modular Approach to $x^5+y^5=z^l$ when $5 \nmid z$}\label{sec:Modular55lBis}

Consider the set
\[L_5:=\{\text{primes } l : 5< l<100\}.\]

\begin{prop}\label{prop:Modular55lBis}
Let $(x,y,z)$ be a primitive integer solution to \eqref{eqn:55l} with $5 \nmid z$ and $l \in L_5$. 
Then \eqref{eqn:ZetaFactor55l} holds with $r=0$.
\end{prop}

The proof is very much analogous to the proof of Proposition~\ref{prop:Modular77l} in Section \ref{sec:Modular77l}. So we just describe the main steps. We use the Frey curve 
\[E_{x,y}:  Y^2= X^3-5(x^2+y^2) X^2 + 5 H_5(x,y) X.\]
Write $\rho_l^{x,y}$ for the Galois representation on the $l$-torsion of $E_{x,y}$. Since $E_{x,y}$ is a quadratic twist of the the Frey curve from Section~\ref{sec:Modular55l} (which is also denoted as $E_{x,y}$ there), the irreducibility of $\rho_l^{x,y}$ for primes $l\geq 7$ follows directly from Lemma~\ref{lem:Irreducibility55l}. Now a straightforward computation of the conductor and minimal discriminant of $E_{x,y}$ and applying modularity \cite{BCDT01} and level lowering \cite{Ribet90}, \cite{Ribet94} as usual, yields the following lemma.

\begin{lem}\label{lem:LevelLowering55lBis}
For a prime $l \geq 7$, the Galois representation $\rho_l^{x,y}$ arises from a newform $f$ of level $N=2^\alpha 5^2$ where $\alpha=1$, $3$, or $4$.
\end{lem}

There are respectively $2$, $5$, and $8$ newforms at these levels, which all happen to be rational.
Specializing $E_{x,y}$ at a trivial primitive integer solution with $xy=0$ (i.e. $(x,y)=(\pm 1, 0)$ or $(0,\pm 1)$ ), yields $E_{200b1}$,  and specializing at $(x,y)=(\pm 1, \pm 1)$ (which does not correspond to a solution) yields $E_{400d2}$. Note that in the latter case we have $H_5(x,y)=1$. 
By comparing traces of Frobenius as usual (including the method of Kraus for some small values of $l$), we can eliminate all but two of the $15$ newforms for all primes $l \geq 7$. The two exceptions being of course the two 
newforms corresponding to the two elliptic curves we just obtained by specialization of $E_{x,y}$. We note that in the case $p|z$ it is convenient to strengthen the congruence $a_p(E_0) \equiv \pm (1+p) \pmod{l}$ to the congruence $a_p(E_0) \equiv a_p(E_{x,y})(1+p) \pmod{l}$.
\begin{lem}\label{lem:ResultAfterCongruences55lBis}
For a prime $l \geq 7$, the Galois representation $\rho_l^{x,y}$ arises from either $E_{200b}$ or $E_{400d}$.
\end{lem}

By a basic application of Kraus' method we are able to eliminate the possibility of  $E_{400d}$ for all $l \in L_5$ except $l=7,11,19$. These remaining three cases can be dealt with using an analogue of Lemma~\ref{lem:ExpSieve77l}.

\begin{lem}
For $l \in L_5$, we have that $\rho_l^{x,y}$ does not arise from $E_{400d}$.
\end{lem}

To finish the proof of Proposition~\ref{prop:Modular55lBis} we now only have to deal with $E_{200b}$, which is possible using again the analogue of Lemma~\ref{lem:ExpSieve77l}. Computational details can be found in the second part of the {\tt MAGMA} script {\tt Modular55l.m}.

\subsubsection{The hyperelliptic curves}{\label{sec:HyperellipticCurves55l}

Now we come to the hyperelliptic curves.

Let $\theta=\zeta+\zeta^{-1}$
and $K=\Q(\theta)$; this is the totally real quadratic subfield of $L$. 
The Galois conjugate of $\theta$ are $\theta_1$, $\theta_2$which in terms of $\zeta$ are given by
\[\theta_1=\zeta+\zeta^{-1}, \qquad
\theta_2=\zeta^2+\zeta^{-2},\]
Note that
\[\theta_1=\theta,\qquad
\theta_2=-1-\theta.\]
Let
\[\gamma=\Norm_{L/K} (\beta).\]
Taking norms in \eqref{eqn:ZetaFactor55l} with $r=0$ down to $K$ we obtain
\[x^2+\theta x y+ y^2= \gamma^l.\]

Let $\gamma_1=\gamma$, $\gamma_2$ denote the conjugates of $\gamma$
that correspond respectively
to $\theta \mapsto \theta_j$, for $j=1,2$. Then
\[x^2+\theta_1 x y+ y^2= \gamma_1^l, \qquad
x^2+\theta_2 x y+ y^2= \gamma_2^l.\]
Furthermore, recall that
\[(x+y)^2=z_1^{2l}.\]
The left hand sides of the previous three equations are symmetric binary quadratic forms over $K$, hence linearly dependent. We calculate
\[(x+y)^2+\theta_2 (x^2+\theta_1 x y+y^2)+\theta_1 (x^2+\theta_2 x y+ y^2) = 0.\]
In terms of the coordinates $z_1^2, \gamma_1, \gamma_2$ we get
\[ z_1^{2l}+\theta_2  \gamma_1^l+\theta_1 \gamma_2^l = 0.\]
Using Remark \ref{rem:FermatQuotient}, we see that 
\[\left(\frac{z_1^2 \gamma_1}{\gamma_2^2},\frac{ z_1^{2l}-\theta_2 \gamma_1^l}{\theta_1 \gamma_2^l}\right) \in C_{l,1}(K)\]
where $C_{l,1}$ is the genus $(l-1)/2$ hyperelliptic curve given by
\begin{equation}\label{eqn:55lC1}
C_{l,1}: Y^2=-4 \eta_1 X^l+1, \quad \eta_1=\theta_2/\theta_1^2=-2\theta-3.
\end{equation}

The linear dependence between the symmetric binary quadratic forms $(x-y)^2$, $(x+y)^2$, and $x^2+\theta xy +y^2$ is given by
\[(\theta-2)(x-y)^2=-4(x^2+\theta xy+y^2)+(\theta+2)(x+y)^2.\]
Using $(x+y)^2=z_1^{2l}$ and $x^2+\theta xy+y^2=\gamma^l$, we get
\begin{equation}\label{eqn:55lC0Pre}
\left(\frac{x-y}{x+y}\right)^2=\frac{-4}{(\theta-2)} \left(\frac{\gamma}{z_1^2}\right)^l+\frac{\theta+2}{\theta-2}.
\end{equation}
We compute $-4/(\theta-2)=4(\theta+3)/5$ and $(\theta+2)/(\theta-2)=-(4\theta+7)/5$. Hence

\[P:=\left( \frac{\gamma}{z_1^2}, \frac{x-y}{x+y}\right) \in C_{l,0}(K)\]
where $C_{l,0}$ is the genus $(l-1)/2$ hyperelliptic curve given by
\begin{equation}\label{eqn:55lC0}
C_{l,0}: 5Y^2= (4 \theta+12) X^l-(4\theta+7).
\end{equation}
Note that in fact the second coordinate of $P$ lies in $\Q$. Furthermore, since $5$ is a square in $K$, the factor $5$ in front of $Y^2$ above could easily be absorbed by rescaling $Y$ (by a factor of $2 \theta+1$). However, this would spoil the nice feature of the curve that the points of our interest have second coordinate lying in $\Q$.

Regarding $K$-rational points on the curves $C_{0,l}$ and $C_{1,l}$, we note that
\begin{equation}\label{eqn:RatPoints55lNumberField}
C_{l,i}(K)  \supset 
\begin{cases}
\{\infty, (1,\pm 1)\} & \text{if } i=0\\
\{\infty, (0,\pm 1), (1,\pm \eta_1)\} & \text{if } i=1.
\end{cases}
\end{equation}

As in (the first part of) Lemma~\ref{lem:ReductionToCurves77l} we have the the following.

\begin{lem}\label{lem:ReductionToCurves55lNumberField}
Let $l \in L_5$. If for $i=0$ or $i=1$ equality holds in~\eqref{eqn:RatPoints55lNumberField}, then there are no non-trivial primitive integer solutions to $x^5+y^5=z^l$ with $5 \nmid z$.
\end{lem}

Let $\p$ be the prime above $5$. We note that $C_{l,1}$ has good reduction at $\p$. Define 
\[C_{l,1}(K)':=\{ P \in C_{l,i}(K) \; : \; P \pmod{\p}=(1,2)\}.\]
As in the $x^7+y^7=z^l$ case, define as well
\[C_{l,0}(K)':=\{ P \in C_{l,0}(K) \; : \; P \text{\ has second coordinate in } \Q\}.\]
Completely similar as in Remark~\ref{rem:ExtraInfo}, we arrive at a refinement of Lemma~\ref{lem:ReductionToCurves55lNumberField}.

\begin{lem}\label{lem:ReductionToCurves55lNumberFieldBis}
Let $l \in L_5$.
\item If for $i=0$ or $i=1$ we have
\[C_{l,i}(K)' = 
\begin{cases}
\{\infty, (1,\pm 1)\} & \text{if } i=0 \\
\{(1,2\theta+3)\} & \text{if } i=1,
\end{cases}\]
then there are no non-trivial primitive integer solutions to $x^5+y^5=z^l$ with $5\nmid z$.
\end{lem}

\subsubsection{Rational points on $C_{l,i}$}

For $i=0,1$ let $J_{l,i}:=\Jac(C_{l,i})$. For $l=11,13,17$ it is easy to check that 
\[[(1,1)-\infty] \in J_{l,0}(K), \quad [(1,\eta_1)-\infty] \in J_{l,1}(K)\]
are points of infinite order. Assume GRH. For these values of $l$ we also computed upper bounds for the ranks of $J_{l,0}(K)$ and $J_{l,1}(K)$; see Table \ref{table:RankBoundsC55l}.

\begin{table}[h!]
\caption{GRH Rank bounds for the Jacobians of $C_{l,0}$ and  $C_{l,1}$}
\begin{tabular}{c||c|c||c|c}
$l$ & $\dim_{\F_2}\Sel^{(2)}(K,\Jac(C_{l,0}))$ & Time  & $\dim_{\F_2}\Sel^{(2)}(K,\Jac(C_{l,1}))$ & Time\\
\hline
11 & 1 & 55s & 2 & 145s\\
13 & 2 & 178s & 1 & 175s\\
17 & 4 & 2178s & 2 & 13087s \\
\end{tabular}
\label{table:RankBoundsC55l}
\end{table}

We conclude that $J_{11,0}(K)$ and $J_{13,1}(K)$ both have rank $1$ and that we have an explicit generator for a finite index subgroup for both of them. Hence, we are again in a position to apply Chabauty-Coleman.

\begin{lem}\label{lem:RatPoints55lBis}
Assuming GRH, we have
\begin{align*}
C_{11,0}(K)' & =  \{\infty, (1,\pm 1)\},\\
C_{13,1}(K)' & = \{(1,2\theta+3)\}.
\end{align*}
\end{lem}

\begin{proof}
We start by determining $C_{11,0}(K)'$ using Chabauty-Coleman with the prime $\p$ above $3$. The curve $C_{11,0}$ has good reduction at $\p$. This reduction, denoted $\tilde{C}_{11,0}$, contains $10$ $\F_9$-rational points, but the subset of $\F_9$-rational points whose second coordinate is $\F_3$-rational consists only of the $4$ points, namely $\infty, (1,\pm 1),(\tilde{X}_0,0)$ where $\tilde{X}_0 \in \F_9$ with $\tilde{X}_0^2=-1$. If we show that for each $\tilde{P}=\infty,(1, \pm 1)$ we have a unique lift to $P \in C_{11,0}(K)$ and that $(\tilde{X}_0,0)$ does not lift to a point in $C_{11,0}(K)$, then it will follow that $C_{11,0}(K)'  =  \{\infty, (1,\pm 1)\}$.

A basis for $\Omega(C_{11,0}/K_{\p})$ is given by $X^i dX/Y$ for $i=0,1,\ldots, 4$. We can compute
\[c_i:=\int_0^D X^i \frac{dX}{Y}, \quad i=0,1, \ldots, 4\]
to high enough $\p$-adic precision and find e.g. that $v_{\p}(c_2)=v_{\p}(c_4)=v_{\p}(c_4-c_2)=v_{\p}(c_4+c_2)=1$. Write $u:=-c_2/c_4$ and let $\omega:=(X^2+uX^4)/dY$. Then we see that $\omega \in \Ann(\Jac(C_{11,0})(K))$ and it reduces to a differential $\tilde{\omega}$ on $\tilde{C}_{11,0}/\F_9$. Since $v_{\p}(u)=v_{\p}(c_2)-v_{\p}(c_4)=0$, we see that $\tilde{\omega}$ does not vanish at $\infty$. Similarly, since $v_{\p}(1+u)=v_{\p}(c_4-c_2)-v_{\p}(c_4)=0$ we see that $\tilde{\omega}$ does not vanish at $(1,\pm 1)$. Finally, since $v_{\p}(-1+u)=v_{\p}(c_4+c_2)-v_{\p}(c_4)=0$ we see that $\tilde{\omega}$ does not vanish at $(\tilde{X}_0,0)$. 
We conclude that for each $\tilde{P}=\infty,(1, \pm 1)$ we have a unique lift to $P \in C_{11,0}(K)$. The point $(\tilde{X}_0,0)$ Hensel-lifts uniquely to a Weierstrass point $(X_0,0) \in C_{11,0}(K_p)$, which is not $K$-rational. This finishes the first part of the proof as in the proof of Proposition~\ref{prop:RatPoints55lCaseI}.

Next we determine $C_{13,1}(K)'$ using Chabauty-Coleman with the prime $\p$ above $5$. The curve $C_{13,1}$ has good reduction at $\p$, denoted $\tilde{C}_{13,1}$. Let $T=X-1$ be a uniformizer at $P=(1,2\theta+3)$. A basis for $\Omega(C_{13,1}/K_{\p})$ is given by $T^i dT/Y$ for $i=0,1,\ldots, 5$. We can compute
\[c_i:=\int_0^D T^i \frac{dT}{Y}, \quad i=0,1, \ldots, 5\]
to high enough $\p$-adic precision and find that $v_{\p}(c_0)=1$ and $v_{\p}(c_i)=2$ for $i=1, \ldots ,5$. This shows that it is impossible to find an $\omega \in \Ann(\Jac(C_{13,1})(K))$ with good reduction at $\p$ which is non vanishing at $\tilde{P} \in \tilde{C}_{13,1}(\F_5)$. Let us define instead $\omega:=(T-c_1/c_5 T^5)dT/Y$. Then $\omega \in \Ann(\Jac(C_{13,1})(K))$ and the reduction mod $\p$ has vanishing order $1$ at $\tilde{P}$. On can indeed check that the Strassmann bound for the function
\[t \mapsto \int_0^{\pi t} (T-\frac{c_1}{c_5} T^5)\frac{dT}{Y}\]
(with $\pi$ a suitable uniformizing parameter) equals $2$. By construction it has a double zero at $t=0$, hence the only lift of $\tilde{P}$  to $C_{13,1}(K)$ is $P$. This means $C_{13,1}(K)' = \{(1,2\theta+3)\}$. Further details can be found in our {\tt MAGMA} script {\tt Chabauty55l.m}.
\end{proof}

We note that it should not be much harder to determine $C_{11,0}(K)$ and $C_{13,1}(K)$ completely. But since it is not necessary for our purposes, we will not pursue this.

\begin{cor}
Assuming GRH, there are no non-trivial primitive integer solutions to~\eqref{eqn:55l} for $l\in \{11,13\}$.
\end{cor}

\subsection{Making the results unconditional}

Full GRH is of course not necessary, we \lq only\rq\ need to obtain certain class and unit group information unconditionally in order to carry out the $2$-descent on the four Jacobians involved. For a hyperelliptic curve defined over a number field  $K$ given by an equation of the form $y^2=f(x)$ where $f(x) \in K[x]$ is irreducible over $K$, it suffices to have available the class and unit group information of the number field $L:=K[x]/f(x)$ (or possibly only certain relative info for the extension $K/L$). For example in the case of $x^5+y^5=z^l$ with $l \in \{11,13\}$ the field $L=L_l$ coming from the curve $C_{11,0}$ for $l=11$ and $C_{13,1}$ for $l=13$ is given by $L_l=\Q[t]/g_l(t)$ with $g_{11}(t):=t^{22}+2t^{11}-4$ and $g_{13}(t)=t^{26}+22 t^{13}-4$. Assuming GRH, either {\tt MAGMA}  or {\tt PARI/GP} can compute the class and unit group info for these two fields rather quickly. In particular, we find that the class group is trivial for both fields (assuming of course GRH). 
It suffices in fact to know that our conditional unit group is a finite index $2$-saturated subgroup of the (unconditional) unit group. This will be easy to check and reduces the problem to verifying that the class groups of the fields $L_{11}$ and $L_{13}$ are trivial. This is something that can be parallelized and it looks like the class group verification for at least $L_{11}$ and probably also $L_{13}$ is within reach of current technology (but the actual verification, especially for $L_{13}$, would in practice of course take considerable effort, time, and computer power).
For $x^7+y^7=z^{11}$ we are looking at number fields of (absolute) degree $33$, and verifying class group information is probably not doable in practice at the moment. The case where $7|z$ might actually be solved using a Hilbert modular approach. We did not pursue this however, since we are not able to treat the case $7 \nmid z$ unconditionally anyway.

Alternatively, we might be able to use partial results on BSD for abelian
varieties over number fields. The four Jacobians $J$ involved, for which we
need to determine the rank unconditionally, all have CM (over a cyclotomic
extension) and are defined over a totally real number field. For such abelian
varieties, the partial BSD result \lq if analytic rank $\leq 1$, then analytic
rank = algebraic rank\rq\ seems within reach; see e.g.\ \cite{Zhang04}. If on
top of this, we are able to compute $L_J(1)$ in the rank $0$ case and $L_J'(1)$
in the three rank $1$ cases to high enough precision to conclude that these
four values are nonzero, then we have made our results unconditional.
However, the computations of $L_J(1)$ and $L_J'(1)$ do not seem to be easier than the class group computations at the moment.

\subsection*{Acknowledgements}
The authors would like to thank the anonymous referee for providing various useful comments and Karim Belabas for discussing class group computations in {\tt PARI/GP}.
The first-named author was supported by a VENI grant from the Netherlands Organisation for Scientific Research (NWO).
The second-named author was supported by an EPSRC Leadership Fellowship.

\end{document}